\documentclass[12pt,reqno]{amsart}
\usepackage{a4wide,amsfonts,amsmath,latexsym,amssymb,euscript,eufrak,graphicx,units,mathrsfs}
\usepackage[utf8]{inputenc}
\usepackage{amsmath}
\usepackage{amsfonts}
\usepackage{amssymb}
\usepackage{amsthm}
\usepackage{floatrow}
\usepackage{blindtext}
\usepackage{multicol}
\usepackage[english]{babel}
\usepackage{enumerate}
\usepackage{eufrak}
\usepackage{graphicx}
\usepackage{caption}
\usepackage{subcaption}
\usepackage{float}
\usepackage{epstopdf}
\usepackage{multirow}
\usepackage{mathtools}
\usepackage{multirow}
\usepackage[usenames,dvipsnames]{color}

\numberwithin{equation}{section}

\newcommand{\R}{\mathbb{R}}
\newcommand{\HC}{\mathcal{H}}
\newcommand{\E}{\mathbb{E}}
\newcommand{\PP}{\mathbb{P}}

\newcommand{\stepid}[1]{\mathbf{1}_{[0, #1]}}

\newcommand{\norm}[1]{\left\lVert#1\right\rVert}
\newcommand{\inner}[2]{\langle #1, #2 \rangle}

\newtheorem{theorem}{Theorem}[section]
\newtheorem{lemma}[theorem]{Lemma}
\newtheorem{proposition}[theorem]{Proposition}
\newtheorem{corollary}[theorem]{Corollary}

\theoremstyle{definition}
\newtheorem{definition}[theorem]{Definition}
\theoremstyle{remark}
\newtheorem{remark}[theorem]{Remark}

\newtheorem{assumption}{Assumption}
\numberwithin{equation}{section}
\begin{document}
\title[Berry-Esseen bounds for Gaussian process estimators]{Berry-Esseen bounds of second moment estimators for Gaussian processes observed at high frequency}
\author[S. Douissi]{Soukaina Douissi}
\address{National School of Applied Sciences, Marrakech, Morocco}
\email{s.douissi@uca.ma}
\author[K. Es-Sebaiy]{ Khalifa Es-Sebaiy}
\address{Department of Mathematics, Faculty of Science, Kuwait University, Kuwait}
\email{khalifa.essebaiy@ku.edu.kw}
\author[G. Kerchev]{George Kerchev}
\address{University of Luxembourg, Department of Mathematics, Luxembourg}
 \email{george.kerchev@uni.lu}
\author[I. Nourdin]{Ivan Nourdin}
\address{University of Luxembourg, Department of Mathematics, Luxembourg}
 \email{ivan.nourdin@uni.lu}

\begin{abstract} Let $Z \coloneqq \{Z_t,t\geq0\}$ be  a stationary Gaussian process. We study two estimators of $\E [Z_0^2]$, namely $\widehat{f}_T(Z)\coloneqq \frac{1}{T} \int_{0}^{T} Z_{t}^{2}dt$, and $\widetilde{f}_n(Z) \coloneqq \frac{1}{n} \sum_{i =1}^{n} Z_{t_{i}}^{2}$, where $ t_{i} = i \Delta_{n}$, $ i=0,1,\ldots, n $, $\Delta_{n}\rightarrow 0$ and  $ T_{n} \coloneqq n \Delta_{n}\rightarrow \infty$. We prove  that the two estimators are strongly consistent and establish Berry-Esseen bounds for a central limit theorem involving $\widehat{f}_T(Z)$ and $\widetilde{f}_n(Z)$. We apply these results to asymptotically stationary Gaussian  processes and estimate the drift parameter   for Gaussian Ornstein-Uhlenbeck processes.
\end{abstract}

\maketitle



\allowdisplaybreaks

\renewcommand{\thefootnote}{\arabic{footnote}}

\section{Introduction}

 Statistical inference for stochastic processes is of great importance for theoreticians and practitioners alike. While for some processes like It\^o-type diffusions and semimartingales, there is extensive literature, the statistical analysis for fractional Gaussian processes is relatively recent.

In this paper, we are interested in the parametric estimation of the variance  of stationary Gaussian process which is not necessarily a semimartingale.  Let  $Z = \{Z_{t},   t \geq0 \}$ be a continuous centered stationary Gaussian process and $f_Z \coloneqq E(Z_0^2)>0$.   We  consider the following estimators of $f_Z$:
\begin{itemize}
\item When a complete path of the process over a large finite interval is observable, we use the estimator:
\begin{equation}\label{estim-cont-STAT}
\widehat{f}_T(Z) \coloneqq \frac{1}{T} \int_{0}^{T} Z_{t}^{2}dt,\quad T>0.
\end{equation}
\item   A more practical assumption is that the process $Z$ is observed at discrete time instants $t_i = i \Delta_n$, where $i = 0, \ldots, n$ and $\Delta_n$ is the step size.  Then we consider the following estimator over the \emph{observation window} $T_n \coloneqq n \Delta n$:
\begin{equation}\label{estim-disc-STAT}
\widetilde{f}_n(Z)  \coloneqq \frac{1}{n} \sum_{i =1}^{n} Z_{t_{i}}^{2},\quad n\geq1.
\end{equation}
\end{itemize}

These estimators are unbiased and we show that they are strongly consistent and admit a central limit theorem. Moreover, we bound the rate of convergence to the normal distribution in terms of total variation distance and Wasserstein distance. Recall that, for two random variables $X$ and $Y$, the former metrics are respectively given by
\begin{align}
\label{eq:def_tv} d_{TV}\left( X,Y\right) \coloneqq \sup_{A\in \mathcal{B}({\mathbb{R}})}\left\vert \PP\left[ X\in A\right] -\PP\left[ Y\in A\right] \right\vert,
\end{align}
where the supremum is over all Borel sets, and
\begin{align}
\label{eq:def_w} d_{W}\left( X,Y\right) \coloneqq \sup_{f\in Lip(1)}\left\vert \E [f(X)]-\E [f(Y)]\right\vert,
\end{align}
where $Lip(1)$ is the set of all Lipschitz functions with Lipschitz constant $\leqslant 1$.

Let $\rho(t) = \rho(-t) \coloneqq \E[Z_0 Z_t]$ for $ t \geq 0$. The central result for $\widehat{f}_T(Z)$, whose proof  follows the lines of the approach developed in \cite[Chapter 7]{NP-book}, is the following.
\begin{theorem}\label{CLT-hat-STAT}  Assume $ \int_\R \rho^2(r) dr  < \infty$. Let   $\mathcal{N} \sim\mathcal{N}(0,1)$ be the standard normal random variable. Then for all $T>0$,
\begin{align}
d_{TV}\left(\frac{\widehat{f}_T(Z) - f_Z}{\sqrt{ Var(\widehat{f}_T(Z) - f_Z)} }, \mathcal{N}\right) \leq \varphi_T(Z),\label{d_TV-of-hat{f}_T}
\end{align}
where
\begin{eqnarray}
\label{eq:varphi_bound}\varphi_T(Z) = C\max\left\{ \frac{8}{\sqrt{T}} \left(\int_{-T}^T |\rho(t)|^{3/2} dt\right)^2,  \frac{48}{T} \left(\int_{-T}^T |\rho(t)|^{4/3} dt\right)^3\right\},
\end{eqnarray}
for some absolute constant $C > 0$. The same result holds for the Wasserstein distance.
\end{theorem}

For the discrete estimator $\widetilde{f}_n(Z)$ we have that:
\begin{theorem}\label{CLT-tilde-STAT}  Assume $ \int_\R \rho^2(r) dr  < \infty$ and that
\begin{align}
\notag \E [ | Z_t - Z_s|^2 ] \leq c |t -s |^{2\alpha}
\end{align}  for some $c>0$ and $\alpha  \in (0,1 )$ and when $|t - s|$ is small enough.
Let   $\mathcal{N} \sim\mathcal{N}(0,1)$ be the standard normal random variable.
If  $  \Delta_n \to 0$ and $n\Delta_n \to \infty$ as $n \to \infty$, then there is $C > 0$ such that, for every $n\geq1$,
\begin{eqnarray*}
 d_{TV}\left(\frac{\widetilde{f}_n(Z) - f_Z}{\sqrt{Var(\widetilde{f}_n(Z) - f_Z) } } , \mathcal{N} \right) \leq \varphi_{T_n}(Z)+  2 \left| 1-  \frac{Var(\widehat{f}_{T_n}(Z) - f_Z) }{Var(\widetilde{f}_n(Z) - f_Z) }\right| + C \left[n\Delta_n^{2\alpha+1}\right]^{1/4},
\end{eqnarray*}
where $\varphi_{T_n}(Z)$ satisfies~\eqref{eq:varphi_bound}. The same result holds for the Wasserstein distance.
\end{theorem}

Thanks to the robustness of our approach, we can extend Theorem~\ref{CLT-hat-STAT} and Theorem~\ref{CLT-tilde-STAT} to the case when the process is only an asymptotically stationary Gaussian process. In particular we prove the rate of convergence of the second moment estimators for $X \coloneqq Z + Y$, where $Z$ is the stationary Gaussian process as above and $Y$ is a stochastic process with $\norm{Y_t}_{L^1} = \mathcal{O}(t^{- \gamma})$ for some absolute $\gamma > 1$. The bounds will be the same up to an extra term $CT^{(1 - \gamma)/2}$ (or $CT_n^{(1 - \gamma)/2}$ for the discrete case).

Moreover, we explore applications to drift estimation for the Ornstein-Uhlenbeck process. Let $X^\theta = (X_t^\theta)_{t \geq 0}$ be an ergodic type Gaussian Ornstein-Uhlenbeck process given by the differential equation
\begin{align}
\notag d X_t^\theta = - \theta X_t^\theta dt + d G_t, \quad X_0^\theta = 0,
\end{align}
where $\theta$ is the drift parameter and $(G_t)_{t \geq 0}$ is an arbitrary mean-zero Gaussian process. One can show that $X^\theta$ is asymptotically stationary and write $X^\theta = Z^\theta + Y^\theta$ where $Z^\theta$ is a stationary Gaussian process and $\norm{Y_t}_{L^1} = \mathcal{O}(t^{- \gamma})$ for any $\gamma > 1$. Then we provide Berry-Esseen bounds for the estimators $\widehat{\theta} \coloneqq g_{Z^\theta} \left( \widehat{f}_T (X^\theta) \right)$ and $\widetilde{\theta} \coloneqq g_{Z^\theta} \left( \widetilde{f}_n (X^\theta) \right)$. The function $g_{Z^\theta}$ is given via $g_{Z^\theta}^{-1} (\theta) = \E [(Z_0^\theta)^2]$ and $\widehat{f}_T$, $\widetilde{f}_{n}$ are as in~\eqref{estim-cont-STAT} and~\eqref{estim-disc-STAT}. Concrete bounds are  computed for the cases when the process $X^\theta$ is of the first and second kind, i.e., when $(G_t)_{t \geq 0}$ is a particular Gaussian process, following the terminology in~\cite{KS}.

Parameter estimation for stationary Gaussian process is usually done via the Maximum likelihood estimator because of its asymptotic optimality, see~\cite{schervish} and~\cite{PDT}. For instance, the MLE estimator of a stationary ARMA process is strongly consistent and asymptotically efficient~\cite[Section 10.8]{BD}.
The method of moments is more computationally tractable especially when one considers discrete estimators. Some recent studies include~\cite{DEV,EV,HS} where the mesh in time $\Delta_n=1$ in~\eqref{estim-disc-STAT}, which is akin to the discretization of a least-squares method for fractional Gaussian processes using fixed-time-step observations. See also~\cite{CTL} for an application of the second moment method to an $AR(1)$ model.

 In the last few years the estimators~\eqref{estim-cont-STAT} and~\eqref{estim-disc-STAT} have been used, in a number of instance, to study parameter estimation problems in various fractional Gaussian models.  Some of these results can be summarized below.
\begin{itemize}
\item \emph{The case of  continuous-time observations for ergodic-type Gaussian processes, using~\eqref{estim-cont-STAT}:}
The work~\cite{SV} derived a central limit theorem and a Berry-Esseen bound in Kolmogorov distance for the second moment estimator~\eqref{estim-cont-STAT}  of   the limiting variance of an
Ornstein-Uhlenbeck (OU) process driven by stationary-increment Gaussian noise. In~\cite{HNZ}, the authors considered the estimator~\eqref{estim-cont-STAT} (called ``Alternative estimator'' there) to estimate the drift parameter of  an  OU process driven by fBm with  Hurst parameter $H\in(0,1)$. They proved a central limit theorem  when   $H\in(0, 3/4]$ and a noncentral limit theorem  for $H\in (3/4, 1)$. However, they did not give speed of convergence for these limit theorems. Their approach used a least
 squares estimator (LSE) to study the second moment estimator. Berry-Esseen bounds for a central limit theorem with the  Kolmogorov distance were given first in~\cite{ES2013} and then further improved upon in~\cite{JLW} when $H \in (1/2, 3/4]$.

Moreover, drift parameter estimation using~\eqref{estim-cont-STAT} was employed for  an  OU process driven  by a Hermite process in~\cite{NT} and driven by an $\alpha-$stable L\'evy motion in~\cite{CHL}.
On the other hand, the consistency and speed of convergence in the TV and Wasserstein norms for the  estimator~\eqref{estim-cont-STAT}  of the drift parameter in infinite dimensional linear stochastic equations driven by a fBm are studied by \cite{KM}.

\item \emph{The case of  discrete-time observations  for ergodic-type Gaussian processes, using
~\eqref{estim-disc-STAT}:} In the case when the mesh in time  $\Delta_n=1$, the consistency and speed of convergence in the TV and Wasserstein distance for the  estimator~\eqref{estim-disc-STAT} of the limiting variance of of general Gaussian sequences were recently developed in the papers \cite{EV, DEV}. Also, the drift parameter in linear stochastic evolution equation driven by a fBm is considered
in \cite{KM}. On the other hand, in the case of high frequency data corresponding to $\Delta_n \to 0$ in~\eqref{estim-disc-STAT}, the statistical inference for several ergodic-type fractional Ornstein-Uhlenbeck (fOU) models, using the estimator~\eqref{estim-disc-STAT}, was recently studied in the papers in~\cite{EEV,HNZ,SV}. However, these papers did not provide speed of convergence for the asymptotic distribution of~\eqref{estim-disc-STAT}.
 \end{itemize}
For the drift parameter estimation for non-ergodic fractional-noise-driven Ornstein-Uhlenbeck processes we refer the interested readers to \cite{EEO,EAA,AAE,EE} and references therein.

For our proofs we employ tools from the analysis on Wiener space, including Wiener chaos calculus and Malliavin calculus. The main theorem regarding normal approximations is the so-called Optimal Fourth Moment Theorem due to Nourdin and Peccati~\cite{NP2015}. A review of these tools and results can be found in Section~\ref{sec:review}. Then the proofs for the stationary and asymptotically stationary Gaussian case are outlined in Sections~\ref{sec:stationary} and~\ref{sec:asym} respectively. The application to drift estimation for Ornstein-Uhlenbeck processes is carried out in Section~\ref{sec:OU}.
Some of our more technical calculations can be seen in Section~\ref{sec:technical}; note that some of results in this section, for instance Proposition~\ref{prop:bound_ft_p}, are of independent interest and use novel techniques.

\section{Elements of Malliavin calculus on Wiener space}\label{sec:review}

This section gives a brief overview of some useful facts from the Malliavin calculus on Wiener space. Some of the results presented here are essential for the proofs in the present paper. For our purposes we focus on special cases that are relevant for our setting and omit the general high-level theory. We direct the interested reader to~\cite[Chapter 1]{nualart-book}and~\cite[Chapter 2]{NP-book}.

Fix $\left( \Omega ,\mathcal{F},\PP \right) $ for the Wiener space of a standard Wiener process $W = (W_t)_{t \geq 0}$. The first step is to identify the general centered Gaussian process $(Z_t)_{\geq 0}$ with an \emph{isonormal Gaussian process} $X = \{ X(h), h \in \mathcal{H}\}$ for some Hilbert space $\mathcal{H}$. Recall that for such processes $X$, for every $h_1, h_2 \in \mathcal{H}$, one has $\E [ X(h_1) X(h_2) ] = \inner{h_1}{h_2}_\HC$.

One can define $\HC$ as the closure of real-valued step functions on $[0, \infty)$ with respect to the inner product
$\inner{\stepid{t}}{\stepid{s}}_\HC = \E[ Z_t Z_s]$. Then the isonormal process $X$ is given by Wiener integral $X(h) \coloneqq \int_{\R^+} h(s) dW_s$. Note, that, in particular $X(\stepid{t}) \overset{d}{=} Z_t$.

The next step involves the \emph{multiple Wiener-It\^o integrals}. The formal definition involves the concepts of Malliavin derivative and divergence. We refer the reader to~\cite[Chapter 1]{nualart-book}and~\cite[Chapter 2]{NP-book}. For our purposes we define the multiple Wiener-It\^o integral $I_p$ via the Hermite polynomials $H_p$. In particular, for $h \in \HC$ with $\norm{h}_\HC = 1$, and any $p \geq 1$,
\begin{align}
\notag H_p(X(h)) = I_p(f^{\otimes p}).
\end{align}
For $p = 1$ and $p = 2$ we have the following:
\begin{align}
\label{eq:z_rep} H_1(X(\stepid{t})) = & X(\stepid{t}) = I_1(\stepid{t}) = Z_t \\
\label{eq:var_rep} H_2(X (\stepid{t})) = & X(\stepid{t})^2 - \E[X(\stepid{t})^2] = I_2(\stepid{t}^{\otimes 2}) = Z_t^2 - \E[Z_t]^2.
\end{align}
Note also that $I_0$ can be taken to be the identity operator.

\begin{remark}\emph{Some notation for Hilbert spaces.} Let $\HC$ be a Hilbert space. Given an integer $q \geq 2$ the Hilbert spaces $\HC^{\otimes q}$ and $\HC^{\odot q}$ correspond to the $q$th \emph{tensor product} and $q$th \emph{symmetric tensor product} of $\HC$. If $f \in \HC^{\otimes q}$ is given by $f = \sum_{j_1, \ldots, j_q} a(j_1, \ldots, j_q) e_{j_1} \otimes \cdots e_{j_q}$, where $(e_{j_i})_{i \in [1, q]}$ form an orthonormal basis of $\HC^{\otimes q}$, then the symmetrization $\tilde{f}$ is given by
\begin{align}
\notag \tilde{f} = \frac{1}{q!} \sum_{\sigma} \sum_{j_1, \ldots, j_q} a(j_1, \ldots, j_q) e_{\sigma(j_1)} \otimes \cdots e_{\sigma(j_q)},
\end{align}
where the first sum runs over all permutations  $\sigma$ of $\{1, \ldots, q\}$. Then $\tilde{f}$ is an element of $\HC^{\odot q}$.
We also make use of the concept of contraction. The $r$th \emph{contraction} of two tensor products $e_{j_1} \otimes \cdots \otimes e_{j_p}$ and $e_{k_1} \otimes \cdots e_{k_q}$ is an element of $\HC^{\otimes (p + q - 2r)}$ given by
\begin{align}
\notag (e_{j_1} & \otimes \cdots \otimes e_{j_p}) \otimes_r (e_{k_1} \otimes \cdots \otimes e_{k_q}) \\
\label{eq:contraction} =  & \quad \left[ \prod_{\ell =1}^r \inner{e_{j_\ell}}{e_{k_\ell}} \right] e_{j_{r+1}} \otimes \cdots \otimes e_{j_q} \otimes e_{k_{r+1}} \otimes \cdots \otimes e_{k_q}.
\end{align}

\end{remark}
The main motivation for introducing the multiple integrals comes from the following properties:
\begin{itemize}
\item \textit{Isometry property of integrals~\cite[Proposition 2.7.5]{NP-book}} Fix integers $p, q \geq 1$ as well as $f \in \HC^{\odot p}$ and $g \in \HC^{\odot q}$.
\begin{align}
\label{eq:isometry}  \E [ I_q(f) I_q(g) ] = \left\{ \begin{array}{ll} p! \inner{f}{g}_{\HC^{\otimes p}} & \mbox{ if } p = q \\ 0 & \mbox{otherwise.} \end{array} \right.
\end{align}
\item \textit{Product formula~\cite[Proposition 2.7.10]{NP-book}} Let $p,q \geq 1$. If $f \in \HC^{\odot p}$ and $g \in \HC^{\odot q}$ then
\begin{align}
\label{eq:product} I_p(f) I_q(g) = \sum_{r = 0}^{p \wedge q} r! {p \choose r} {q \choose r} I_{p + q -2r}(f \widetilde{\otimes}_r g).
\end{align}
\item \textit{Hypercontractivity in Wiener Chaos.} For every $q\geq 1$, ${\mathcal{H}}_{q}$ denotes the $q$th Wiener chaos of $W$, defined as the closed linear subspace of $L^{2}(\Omega )$ generated by the random variables $\{H_{q}(W(h)),h\in {{\mathcal{H}}},\Vert h\Vert _{{\mathcal{H}}}=1\}$ where $H_{q}$ is the $q$th Hermite polynomial. For any $F \in  \oplus_{l=1}^{q}{\mathcal{H}}_{l}$ (i.e. in a fixed sum of Wiener chaoses), we have
\begin{equation}
\left( \E\big[|F|^{p}\big]\right) ^{1/p}\leqslant c_{p,q}\left( \E\big[|F|^{2}\big]\right) ^{1/2}\ \mbox{ for any }p\geq 2.
\label{hypercontractivity}
\end{equation}
It should be noted that the constants $c_{p,q}$ above are known with some precision when $F$ is a single chaos term: indeed, by~\cite[Corollary 2.8.14]{NP-book}, $c_{p,q}=\left( p-1\right) ^{q/2}$.

\end{itemize}

The second part of important results we borrow from Malliavin calculus concerns estimates on the distance between random variables. There are two key estimates linking total variation distance and the Malliavin calculus, which were both obtained by Nourdin and Peccati. The first one is an observation relating an integration-by-parts formula on Wiener space with a classical result of Ch. Stein. The second is a quantitatively sharp version of the famous fourth moment theorem of Nualart and Peccati.

Let $\mathcal{N}$ denote the standard normal law. For each integer $n$, let $F_{n}\in {\mathcal{H}}_{q}$. Assume $Var\left[ F_{n}\right] =1$ and $\left( F_{n}\right) _{n}$ converges in distribution to a normal law. It is known (the \emph{fourth moment theorem} in~\cite{NP}) that this convergence is equivalent to $\lim_{n}\E\left[ F_{n}^{4}\right] =3$. The following optimal estimate for $d_{TV}\left( F_n,\mathcal{N}\right) $, known as the optimal fourth moment theorem, was proved in \cite{NP2015}: with the sequence $F$ as above, assuming convergence, there exist two constants $c,C>0$ depending only on the type of the process $F$ but not on $n$, such that
\begin{equation}
c\max \left\{ \E\left[ F_{n}^{4}\right] -3,\left\vert \E\left[ F_{n}^{3}\right] \right\vert \right\} \leqslant d_{TV}\left(F_{n},\mathcal{N}\right) \leqslant C\max \left\{ \E\left[F_{n}^{4}\right] -3,\left\vert \E\left[ F_{n}^{3}\right]\right\vert \right\} . \label{optimal berry esseen}
\end{equation}

Recall that for a standardized random variable $F$, i.e., with $\E[F] = 0$ and $\E[F^2] = 1$, the third and fourth cumulants are respectively
\begin{align}
\notag \kappa_3 (F) \coloneqq & \E[F^3], \\
\notag  \kappa _{4}\left( F\right)  \coloneqq & \E\left[ F^{4}\right] -3.
\end{align}
Throughout the paper we use the notation $\mathcal{N} \sim \mathcal{N}(0,1)$. We also use the notation $C$ for any positive real constant, independently of its value which may change from line to line when this does not lead to ambiguity.
\begin{remark}\label{rem:wass_tv} We note that the optimal bound~\eqref{optimal berry esseen} holds with $d_{TV}$ replaced by $d_W$. Indeed, by~\cite[Theorem 3.5.2]{NP-book}:
\begin{align}
\notag d_W(F,\mathcal{N} )\leq \sup_{\varphi\in \mathcal{F}} |\E[\varphi'(F)]-\E[F\varphi(F)]|
\end{align}
where $\mathcal{F}$ is the set of $C^1$ functions $\varphi$ such that $|\varphi'|_\infty \leq \sqrt{2/\pi}$. Now using the concepts of the Malliavin derivative operator $D$ and Ornstein-Uhlenbeck generator $L$, see~\cite[Theorem 4.15]{Nourdin2013}, one has $|\E[\varphi'(F)]-\E[F\varphi(F)]|  = | \E[ \varphi'(F) \E[(1-\langle DF,-DL^{-1}F\rangle) | \mathcal{F} ] ] |$, and thus:
\begin{align}
\notag d_W(F,\mathcal{N} )\leq \sqrt{2/\pi} \E |\E[1-\langle DF,-DL^{-1}F \rangle | \mathcal{F} ] |
\end{align}
However, this is the same bound one has (up to the constant $\sqrt{2/\pi}$) in the proof of the optimal fourth moment theorem~\cite[Proof of Theorem 1.2]{NP2015}.

\end{remark}

\section{Parameter estimation for stationary Gaussian processes}\label{sec:stationary}

In this section we present a general framework for the parameter estimation  of the variance  of a stationary Gaussian process.
 We prove the consistency and provide upper bounds in the total variation and Wasserstein distances  for the rate of normal
convergence of the MC estimators (\ref{estim-cont-STAT}) and (\ref{estim-disc-STAT}).\\

 Let $Z \coloneqq \{Z_{t},   t\geq0 \}$ be a continuous centered stationary Gaussian process that can be represented as a Wiener-It\^o (multiple) integral $Z_{t} = I_{1}(\stepid{t})$ for every $t\geq 0$, as  in~\eqref{eq:z_rep}.  Let $\rho(r)=E(Z_rZ_0)$ denote the covariance of $Z$ for every $r\geq0$, and  let $\rho(r)=\rho(-r)$ for all $r<0$. Our main assumption throughout the paper is that
\begin{align}
\notag \sigma_Z^2 \coloneqq 4 \int_{\R}\rho^2(r)dr<\infty.
\end{align}

\subsection{Continuous-time observations}
We estimate the variance $f_Z \coloneqq \E(Z_0^2)$, when the whole trajectory  of $Z$ is observed up to time $T> 0$. We consider the estimator~\eqref{estim-cont-STAT} given by
\begin{equation*}
\widehat{f}_T(Z)=\frac{1}{T} \int_{0}^{T} Z_{t}^{2}dt,\quad T>0
\end{equation*}
as a statistic to estimate   $f_Z$, based on the continuous-time observation of $Z$. Our goal is to establish Theorem~\ref{CLT-hat-STAT}, i.e.,  a Berry-Esseen bound on the the convergence of $\widehat{f}_T(Z) - f_Z$. First we show some simpler properties of $\widehat{f}_T(Z)$.

\begin{lemma}\label{lem:ft_strongc} The estimator $\widehat{f}_T(Z)$ is unbiased and strongly consistent. In particular,
\begin{align}
\label{eq:VT2_finite} \sqrt{T} \norm{\widehat{f}_T(Z) - f_Z}_{L^2} \uparrow \sigma_Z \mbox{ as }  T\to \infty.
\end{align}
\end{lemma}

\begin{proof}
By stationarity, $\E [Z_0^2] = \E[Z_s^2]$ for every $s \geq 0$ and thus $\E[ \widehat{f}_T(Z) ] = f_Z$, so the estimator $\widehat{f}_T(Z)$ is unbiased.

The next step is to show that the estimator $\widehat{f}_T(Z)$ is strongly consistent, i.e.,
\begin{align}
\notag \widehat{f}_T(Z) \to f_Z \mbox{ almost surely as }T \to \infty.
\end{align}
 Let
\begin{align}
\label{chaos-V} V_{T}(Z) \coloneqq\sqrt{T} \left(\widehat{f}_T(Z) - f_Z\right) =  \frac{1}{\sqrt{T}} \int_{0}^{T} \left( Z_{t}^{2} - E[Z_{t}^{2}] \right) dt =\frac{1}{\sqrt{T}} \int_{0}^{T} I_2(\stepid{t}^{\otimes 2}) dt,
\end{align}
where we have used~\eqref{eq:var_rep}. Then
\begin{align}
\label{eq:second_exp} \E[V_T(Z)^2] =  \frac{1}{T} \E \left[ \left( \int_0^T I_2(\stepid{t}^{\otimes 2}) dt \right)^2 \right] = \frac{1}{T} \int_{[0, T]^2} \E \left[ I_2(\stepid{t}^{\otimes 2}) I_2(\stepid{s}^{\otimes 2})\right] dt ds
\end{align}
By~\eqref{eq:isometry},
\begin{align}
\notag  \E \left[ I_2(\stepid{t}^{\otimes 2}) I_2(\stepid{s}^{\otimes 2})\right] = 2! \inner{\stepid{t}^{\otimes 2}}{\stepid{s}^{\otimes 2}}_{\HC^{\otimes 2}} = 2 \left( \E[ Z_t Z_s] \right)^2 = 2 \rho(t-s)^2.
\end{align}
Therefore,
\begin{align}
\notag \E[V_T(Z)^2] =  \frac{2}{T} \int_{[0, T]^2} \rho(t-s)^2 dt ds = \frac{4}{T} \int_0^T \int_u^T \rho(u)^2 dt du \\
\label{eq:vt_conv} = 4 \int_0^\infty \left(1 - \frac{u}{T}\right)_+ \rho(u)^2 \leq 4 \int_0^T \rho(u)^2 du < \infty.
\end{align}
where $(x)_+ = \max \{ x, 0\}$. Note that the above also implies that $\E [V_T(Z)^2] \uparrow \sigma_Z^2$, establishing~\eqref{eq:VT2_finite}. Alternatively,  $\norm{\widehat{f}_T(Z)  - f_Z}_{L^2} \leq \sigma_Z/ \sqrt{T}$. In particular, this shows that $\widehat{f}_T(Z)$ converges to $f_Z$ in $L^2$. At this point we recall~\cite[Lemma 2.1]{KN}):
\begin{lemma}
\label{Borel-Cantelli} Let $\gamma >0$. Let $(Z_{n})_{n\in \mathbb{N}}$ be a sequence of random variables. If for every $p\geq 1$ there exists a constant $c_{p}>0$ such that for all $n\in \mathbb{N}$,
\begin{equation*}
\Vert Z_{n}\Vert _{L^{p}(\Omega )}\leqslant c_{p}\cdot n^{-\gamma },
\end{equation*}
then for all $\varepsilon >0$ there exists a random variable $\alpha _{\varepsilon }$ which is almost surely finite such that
\begin{equation*}
|Z_{n}|\leqslant \alpha _{\varepsilon }\cdot n^{-\gamma +\varepsilon }\quad \mbox{almost surely}
\end{equation*}
for all $n\in \mathbb{N}$. Moreover, $\E|\alpha _{\varepsilon}|^{p}<\infty $ for all $p\geq 1$.
\end{lemma}
We can apply Lemma~\ref{Borel-Cantelli} as soon as $\norm{\widehat{f}_T(Z)  - f_Z}_{L^p} \leq C_p / \sqrt{T}$ for every $p \geq 1$.  We have shown that  for $p = 2$. However, using the inequality between $L^p$ norms and $L^q$ norms, one has that the same bound holds for $p \in [1, 2)$. Finally, one can apply the hypercontractivity property~\eqref{hypercontractivity} in Wiener chaos to get the result for all $p \geq 2$.

Therefore, $\widehat{f}_n(Z)$ converges almost surely to $f_Z$ as $n \to \infty$ (and $n \in \mathbb{N})$. We now use the following more technical result. Its proof is delayed to the Section~\ref{sec:technical}.

\begin{lemma}\label{continuous-Borel-Cantelli} Let $\{u_t, t\geq0\}$ be a continuous stochastic process such that for any $p\geq1$, there is a positive constant $C_p>0$ such that $\sup_{t\geq0}\E[|u_t|^p]<C_p$. In addition, we assume
 \begin{align}
\notag \frac{1}{n}\int_0^n u_t dt\longrightarrow0\quad \mbox{almost surely as } n\rightarrow\infty.
\end{align}
 Then,
 \[\frac{1}{T}\int_0^T u_t dt\longrightarrow0\quad \mbox{almost surely as } T\rightarrow\infty.\]
\end{lemma}
Note that by stationarity $\sup_{t \geq 0} \E[|Z_t|^p] = \E[|Z_0|^p < C_p$ for some $C_p > 0$. Therefore, by Lemma~\ref{continuous-Borel-Cantelli}, $\widehat{f}_T(Z)$ converges almost surely to $f_Z$ as $T \to \infty$ (and $T \in \R_+$).
\end{proof}

Now, we turn to the proof of Theorem~\ref{CLT-hat-STAT}.
\subsection{Proof of Theorem~\ref{CLT-hat-STAT}.} The random variable
\begin{align}
\notag \frac{\widehat{f}_T(Z) - f_Z}{\sqrt{Var(\widehat{f}_T(Z) - f_Z)}} = \frac{V_T(Z)}{ \sqrt{\E [V_T(Z)^2]}}
\end{align}
 is centered and normalized.  Then the fourth moment theorem~\eqref{optimal berry esseen} applies and then
\begin{align}
\notag d_{TV}\left(\frac{\widehat{f}_T(Z) - f_Z}{\sqrt{Var(\widehat{f}_T(Z) - f_Z)}} ,\mathcal{N}\right) = & \quad d_{TV} \left(\frac{V_T(Z)}{ \sqrt{\E [V_T(Z)^2]}}, \mathcal{N} \right)\\
\label{eq:bound_tv}  \leq & \quad C\max \left\{ \frac{\kappa_3
(V_T(Z))}{\E[V_T(Z)^2]^{3/2}},
\frac{\kappa_4(V_T(Z))}{\E[V_T(Z)^2]^2} \right\}.
\end{align}
We are left to study the third and fourth cumulants of $V_T(Z)$. The following technical result is a slight modification of~\cite[Propositions 6.3, 6.4]{BBNP}:
\begin{lemma}\label{rate-cumulants} For every $T>0$,
\begin{align}
\label{third-cumulant} \left|\kappa_{3}(V_{T}(Z))\right| \leq & \quad \frac{8}{\sqrt{T}}\left( \int_{-T}^{T} \left| \rho(t)\right|^{3/2}dt \right)^{2}, \\
 \label{fourth-cumulant} \left|\kappa_{4}(V_{T}(Z))\right|\leq & \quad  \frac{48}{T} \left( \int_{-T}^{T} \left| \rho(t)\right|^{4/3} dt\right)^{3}.
\end{align}
\end{lemma}

\begin{proof}
The proof follows the same approach as in~\cite{BBNP} and is included in Section~\ref{sec:technical}. We  note that our proof is more detailed than the one in~\cite{BBNP}.
\end{proof}

The corresponding bound for the Wasserstein distance follows from Remark~\ref{rem:wass_tv}.
Thus Theorem~\ref{CLT-hat-STAT} is established. At this point we present two corollaries. First, we study the asymptotic behavior of the bound~\eqref{eq:varphi_bound} under the additional assumption of the decay of correlations of $Z_t$. We have the following:

\begin{corollary}\label{cor:explicit_bounds} Assume  there exists  $0<\beta<\frac34$ such that, $|\rho(t)| = \mathcal{O}(t^{2 \beta -2})$. Then,
\begin{align}
\label{eq:bound_tv_explicit} d_{TV}\left(\frac{\widehat{f}_T(Z) - f_Z}{\sqrt{ Var(\widehat{f}_T(Z) - f_Z)} }, \mathcal{N}\right) \leq C \left\{ \begin{array}{ll}
T^{-1/2} & \mbox{ if }0<\beta <\frac{2}{3}, \\
\\
\log^2(T) T^{-1/2} & \mbox{ if }\beta =\frac23, \\ ~~ &  \\
T^{6\beta-\frac92} & \mbox{ if }\frac23<\beta<\frac34.
\end{array}%
\right.
\end{align}

\end{corollary}

\begin{proof}Note that there is a constant $C>0$ such that $|\rho(s)| < C$ for $s \in [-1,1]$. Then, since
\begin{align}
\notag \int_{-T}^T |\rho(t)|^p dt \leq C_p \left( 1 + \int_1^T |\rho(t)|^p dt \right),
\end{align}
for $p =3/2$ and $p = 4/3$, one can establish the following bounds on $\kappa_3(V_T(Z))$ and $\kappa_4(V_T(Z))$:
\begin{eqnarray}
\left|\kappa_{3}(V_{T}(Z))\right| & \leq   C \left\{ \begin{array}{ll} T^{-1/2} & \mbox{ if }0<\beta <\frac{2}{3}, \\
\\
\log^2(T) T^{-1/2} & \mbox{ if }\beta =\frac23, \\ ~~ &  \\
T^{6\beta-\frac92} & \mbox{ if }\frac23<\beta<\frac34,
\end{array}%
\right.\label{explicit-third-cumulant}
\end{eqnarray}
and
 \begin{eqnarray}
 \left|\kappa_{4}(V_{T}(Z))\right| & \leq   C \left\{ \begin{array}{ll}
T^{-1} & \mbox{ if }0<\beta <\frac58, \\
\\
\log^3(T) T^{-1} & \mbox{ if }\beta =\frac58, \\ ~~ &  \\
T^{8\beta-6} & \mbox{ if }\frac58<\beta<\frac34.
\end{array}%
\right.\label{explicit-fourth-cumulant}
\end{eqnarray}
The result follows by a direct application of Theorem~\ref{CLT-hat-STAT}.
\end{proof}

Next, using the convergence of $\E[V_T(Z)^2]$, one can also establish the following corollary to Theorem~\ref{CLT-hat-STAT}.

\begin{corollary}\label{cor:sigmaz}  There exists a constant $C>0$ such that, for all $T>0$,
\begin{align}
\notag d_{TV}\left(\frac{\sqrt{T}}{\sigma_Z} (\widehat{f}_T(Z) - f_Z), \mathcal{N} \right) \leq \varphi_T(Z) + 2\left| 1- \frac{\sigma_Z^2}{E(V_{T}(Z)^2)}\right|
\end{align}
where $\varphi_T(Z)$ is as in Theorem~\ref{CLT-hat-STAT}.
Moreover, if  there exists  $0<\beta<\frac34$ such that, $|\rho(t)| = \mathcal{O}(t^{2 \beta -2})$,
\begin{align}
\label{eq:explicit_sz} d_{TV}\left(\frac{\sqrt{T}}{\sigma_Z} (\widehat{f}_T(Z) - f_Z), \mathcal{N} \right) \leq C \left\{ \begin{array}{ll}
T^{-1/2} & \mbox{ if }0<\beta \leq \frac58, \\ ~~ &  \\
T^{4\beta-3} & \mbox{ if }\frac58<\beta<\frac34.
\end{array} \right.
\end{align}
The same result holds for the Wasserstein distance.
\end{corollary}

\begin{proof} The first part follows from the following technical result:

\begin{lemma}[{\cite[Lemma 5.1]{DEV}}]\label{upper-d_TV-lemma}Let  $\mu \in \R$ and $\sigma>0$. Then , for every integrable real-valued random variable $F$,
\begin{align}
\notag d_{TV}\left( \mu+\sigma F,\mathcal{N}\right) \leq d_{TV}\left( F,\mathcal{N} \right)+ \sqrt{\frac{\pi}{2}}|\mu|+2\left|1-\frac{1}{\sigma^2}\right|.
\end{align}
\end{lemma}
Therefore,
\begin{align}
\notag d_{TV}\left(\frac{\sqrt{T}}{\sigma_Z} (\widehat{f}_T(Z) - f_Z), \mathcal{N} \right) \leq d_{TV}\left(\frac{\widehat{f}_T(Z) - f_Z}{\sqrt{Var(\widehat{f}_T(Z) - f_Z)}}, \mathcal{N} \right) + 2 \left| 1 -  \frac{\sigma_Z^2}{\E[ V_T(Z)^2] }\right|,
\end{align}
by the definition of $V_T(Z)$ and the fact that $\E[ \widehat{f}_T(Z)] = f_Z$. Recall that $\E[V_T(Z)^2] \uparrow \sigma_Z^2$. Then the second term on the right-hand side above is bounded by $C |\E[V_T(Z)^2 - \sigma_Z^2|$ for some $C > 0$.

Next, since $|\rho(t)| = \mathcal{O}(t^{2 \beta -2})$, and using the representation~\eqref{eq:vt_conv},
\begin{align}
\notag | \E[ V_T(Z)^2 - \sigma_Z^2| = & \quad 4 \int_T^\infty \rho(u)^2 du + 4 \int_0^T \frac{u}{T} \rho(u)^2 du \\
\notag \leq & \quad 4C \left( \int_T^\infty u^{4 \beta - 4} du + \frac{1}{T} \int_0^1 \rho(u)^2 du + \frac{1}{T} \int_1^T u^{4 \beta - 3} du \right).
\end{align}
A direct computation yields that,
\begin{align}
\label{eq:bound_vt_sz} | \E[ V_T(Z)^2 - \sigma_Z^2|  \leq C \left\{
\begin{array}{ll}
T^{-1} & \mbox{ if }0<\beta <\frac{1}{2}, \\
\\
\log(T) T^{-1} & \mbox{ if }\beta =\frac12, \\
~~ &  \\
T^{4\beta-3} & \mbox{ if }\frac12<\beta<\frac34,
\end{array}
\right.
\end{align}
Finally, the bound~\eqref{eq:explicit_sz}  follows from~\eqref{eq:bound_tv_explicit} and~\eqref{eq:bound_vt_sz}.

\end{proof}

\subsection{Discrete-time observations}
In this section we estimate the limiting variance $f_Z$ based on discrete high-frequency data in time of $Z$, by considering the discrete version $\widetilde{f}_n(Z)$ of the estimator $\widehat{f}_T(Z)$:
\begin{align}
\notag \widetilde{f}_n(Z) \coloneqq \frac{1}{n} \sum_{i =1}^{n} Z_{t_{i}}^{2},
\end{align}
where $t_{i} = i \Delta_{n}$, $  i=0,\ldots, n $, $\Delta_{n}\rightarrow0$ and  $ T_{n} \coloneqq n \Delta_{n}\rightarrow \infty$.

We will assume an additional property of the process $Z$. It would allow us to compare quantitatively the two estimators $\widehat{f}_T(Z)$ and $\widetilde{f}_n(Z)$:

\begin{assumption}\label{assume:helix} For all $s,t \in \R_+$ such that $|s-t|$ is small enough,
\begin{eqnarray}\label{helix-property}
\E[|Z_{t}-Z_{s}|^{2}] \leq C |t-s|^{2 \alpha},
\end{eqnarray}
for some   constant $0<\alpha<1$.
\end{assumption}

Similarly to the continuous alternative $\widehat{f}_T(Z)$, we first show the following:

\begin{lemma} The estimator $\widetilde{f}_n(Z)$ is unbiased. Assume that Assumption~\ref{assume:helix} holds. Then,
\begin{align}
\label{estimate-difference}  \E |\widehat{f}_T(Z) - \widetilde{f}_n(Z)|^2 \leq C_\alpha \Delta_n^{2 \alpha},
\end{align}
where $C_\alpha > 0$ is a constant that depends only on $\alpha$. Moreover, if $n \Delta_n^{\eta} \to 0, \mbox {as } n \to \infty$ for some $\eta > 1$, then $\widetilde{f}_n(Z)$ is strongly consistent.
\end{lemma}

\begin{proof} The first property follows from the fact that the process $Z$ is stationary. To show strong consistency define, similarly to~\eqref{chaos-V},
\begin{align}
\label{eq:def_U} U_{n}(Z) \coloneqq \sqrt{T_{n}} \left( \widetilde{f}_n(Z) - f_Z\right) = \frac{1}{\sqrt{T_{n}}} \sum_{i =1}^{n} (Z_{t_{i}}^{2}-\E Z_{t_{i}}^{2}).
\end{align}
Then
\begin{align}
\label{eq:diff} \widetilde{f}_n(Z) - f_Z =  \frac{U_n(Z)}{\sqrt{T_n}} = \frac{V_{T_n}(Z)}{\sqrt{T_n}} + \frac{U_n(Z) - V_{T_n}(Z)}{\sqrt{T_n}}.
\end{align}
We know from~\eqref{eq:VT2_finite} and Lemma~\ref{Borel-Cantelli} that $V_{T_n} / \sqrt{T_n}$ converges almost surely to $0$. Let $\delta_n(Z) \coloneqq U_n(Z) - V_{T_n}(Z) = \sqrt{T_n} (\widetilde{f}_n(Z) - \widehat{f}_T(Z))$. We estimate the second moment of $\delta_n(Z)$:
\begin{align}
\notag \E[ \delta_n(Z)^2] \leq   \E \left[ \left(\frac{1}{\sqrt{T_n}} \sum_{i = 1}^n \int_{t_{i-1}}^{t_i} |Z_{t_i}^2 - Z_t^2| dt \right)^2 \right]  \leq   \frac{n}{T_n} \sum_{i = 1}^n \E \left[ \left( \int_{t_{i-1}}^{t_i}|Z_{t_i}^2 - Z_t^2| dt \right)^2 \right],
\end{align}
where we have applied the inequality between arithmetic mean and quadratic mean. Next, by the Cauchy-Schwarz inequality, for $s,t \in [t_{i-1}, t_i]$,
\begin{align}
\notag \E[ |Z_{t_i}^2 - Z_t^2| |Z_{t_i}^2 - Z_s^2| ] \leq & \quad  \left( \E [ (Z_{t_i}^2 - Z_t^2)^2  ] \E [ (Z_{t_i} - Z_s)^2 ]  \right)^{1/2} \\
\notag  \leq & \quad  \left( \E [ (Z_{t_i} - Z_t)^4  ] \E [ (Z_{t_i} - Z_s)^4 ] \E [ (Z_{t_i} + Z_t)^4  ] \E [ (Z_{t_i} + Z_s)^4 ]  \right)^{1/4}.
\end{align}
Now, by the hypercontractivity property~\eqref{hypercontractivity} and~\eqref{helix-property}
\begin{align}
\notag  \E [ (Z_{t_i} - Z_t)^4]^{1/4} \leq C \E [ (Z_{t_i} - Z_t)^2]^{1/2}  \leq C |t_i - t|^{ \alpha}.
\end{align}
Next, using stationarity
\begin{align}
\notag \E[ (Z_{t_i} + Z_t)^4 ] \leq 4 \E[ (Z_{t_i}^2 + Z_t)^2 ] \leq 16 \E[Z_0^4].
\end{align}
 Therefore, for some constant $C > 0$ depending on the fourth moments of $Z_0$,
\begin{align}
\notag \E[ \delta_n(Z)^2] \leq  & \quad \frac{n}{T_n} \sum_{i = 1}^n \int_{t_{i-1}}^{t_i} \int_{t_{i-1}}^{t_i} C |t_i - t|^{\alpha} |t_i - s|^{\alpha} ds dt  \\
\notag = & \quad \frac{C \Delta_n^{2\alpha + 2}}{T_{n}} \sum_{i=1}^{n} \int_{0}^{1}  \int_{0}^{1} (uv)^{\alpha}du dv\\
\notag \leq  & \quad C n\Delta_n^{2\alpha+1},
\end{align}
with  $u=\frac{t-t_{i-1}}{\Delta_n}$, $v=\frac{s-t_{j-1}}{\Delta_n}$. Thus,~\eqref{estimate-difference}  is established.

Now, using that $n \Delta_n^\eta \to 0$ for some $\eta > 1$,
\begin{align}
\E \left[ \left( \frac{\delta_{n}(Z)}{\sqrt{T_n}} \right)^2 \right] \leq C \Delta_n^{2\alpha} \leq  C n^{-2\alpha/\eta}\left(n\Delta_n^{\eta}\right)^{2\alpha/\eta} \leq C n^{-2\alpha/\eta}.
\end{align}
By the hypercontractivity property~\eqref{hypercontractivity} and Lemma~\ref{Borel-Cantelli}, we obtain
$ \frac{\delta_{n}(Z)}{\sqrt{T_n}} \to 0,$ and thus $\widehat{f}_{T_n}(Z) - \widetilde{f}_n(Z) \to 0$ almost surely. By Lemma~\ref{lem:ft_strongc} $\widehat{f}_{T_n}  - f_Z \to 0$ almost surely. Therefore, $\widetilde{f}_n(Z) - f_Z \to 0$ almost surely and  the discrete estimator is strongly consistent.

\end{proof}

We now turn to the proof of Theorem~\ref{CLT-tilde-STAT}.

\subsection{Proof of Theorem~\ref{CLT-tilde-STAT}.}
The discrete estimator $\widetilde{f}_n(Z)$ is unbiased, so $Var(\widetilde{f}_n(Z) - f_Z) = \E [( \widetilde{f}_n(Z) - f_Z)^2]$. Next, by the triangle inequality one has,
\begin{align}
\notag d_{TV} \left( \frac{\widetilde{f}_n(Z)  - f_Z}{\sqrt{ \E [( \widetilde{f}_n(Z) - f_Z)^2] }}, \mathcal{N}\right) \leq \quad &  d_{TV} \left( \frac{\widetilde{f}_n(Z)  - f_Z}{\sqrt{ \E [( \widetilde{f}_n(Z) - f_Z)^2] }}, \frac{\widehat{f}_{T_n}(Z)  - f_Z}{\sqrt{ \E [( \widetilde{f}_n(Z) - f_Z)^2] }} \right) \\
\label{eq:tv_triangle} &+  d_{TV} \left( \frac{\widehat{f}_n(Z)  - f_Z}{\sqrt{ \E [( \widetilde{f}_{T_n}(Z) - f_Z)^2] }}, \mathcal{N}\right) .
\end{align}
By Lemma~\ref{upper-d_TV-lemma},
\begin{align}
\notag  d_{TV} \left( \frac{\widehat{f}_n(Z)  - f_Z}{\sqrt{ \E [( \widetilde{f}_{T_n}(Z) - f_Z)^2] }}, \mathcal{N}\right) \leq d_{TV} \left( \frac{\widehat{f}_n(Z)  - f_Z}{\sqrt{ \E [( \widehat{f}_{T_n}(Z) - f_Z)^2] }}, \mathcal{N}\right) + 2 \left| 1 - \frac{Var(\widehat{f}_{T_n}(Z) - f_Z)}{Var(\widetilde{f}_n(Z) - f_Z)} \right|.
\end{align}
Then, by Theorem~\ref{CLT-hat-STAT} the first term on the right-hand side is further bounded by $\varphi(T_n)$ as defined in~\eqref{eq:varphi_bound}.

Thus, we are left to bound the first term in~\eqref{eq:tv_triangle}. By the definition of total variation distance one has
\begin{align}
\notag d_{TV} \left( \frac{\widetilde{f}_n(Z)  - f_Z}{\sqrt{ \E [( \widetilde{f}_n(Z) - f_Z)^2] }}, \frac{\widehat{f}_{T_n}(Z)  - f_Z}{\sqrt{ \E [( \widetilde{f}_n(Z) - f_Z)^2] }} \right) = d_{TV} (U_n(Z), V_{T_n}(Z) ),
\end{align}
with the two quantities given in~\eqref{chaos-V} and~\eqref{eq:def_U}. We recall the following technical result
\begin{lemma}[{\cite[Theorem 3.5]{kosov}}]\label{dTV}   There is a positive constant $C>0$ such that, for all multiple integrals $F$ and $G$ of order 2,
\begin{eqnarray*}
d_{TV}\left(F,G\right) \leq C \left( \frac{ \E [ (F- G)^2]}{\E[ F^2 ] } \right)^{1/4}.
\end{eqnarray*}
\end{lemma}

Therefore, there is a positive constant $C>0$ such that, for every $n \geq 1$
\begin{align}
\label{d_TV-continuous-discrete} d_{TV}\left(U_{n}(Z),V_{T_n}(Z)\right) \leq C \left( \frac{ \E[ \delta_n(Z)^2 ] }{\E [ V_{T_n}(Z)^2 ] } \right)^{1/4} \leq C (n \Delta n^{2 \alpha + 1})^{1/4},
\end{align}
where we have used~\eqref{eq:VT2_finite} and~\eqref{estimate-difference}. The corresponding bound for the Wasserstein distance follows from Remark~\ref{rem:wass_tv}. Thus, the proof of Theorem~\ref{CLT-tilde-STAT} is completed.

\begin{remark} Let us present an example where the bounds on total variation in Theorem~\ref{CLT-tilde-STAT} decrease to $0$ and thus a central limit theorem holds. Let $\Delta_n=1/n^{\lambda}$ with $\frac{1}{2\alpha+1}<\lambda<1$. Then, $n\Delta_n\rightarrow\infty$ and $n\Delta_n^{2\alpha+1}\rightarrow0$. Moreover,
\begin{align}
\notag 2 \left| 1 - \frac{Var(\widehat{f}_{T_n}(Z) - f_Z)}{Var(\widetilde{f}_n(Z) - f_Z)} \right| = 2 \left| \frac{\E[ U_n(Z)^2] - \E[ V_{T_n}(Z)^2] }{\E[ U_n(Z)^2]}\right| \leq 2 \frac{\E[\delta_n(Z)^2]}{\E[U_n(Z)^2]}.
\end{align}
Now, if $\E[\delta_n(Z)^2] \to 0$ and $\E[V_{T_n}(Z)] \uparrow \sigma_Z^2$, then $\E[U_n(Z)^2 ] \uparrow \sigma_Z^2$. Thus,
\begin{align}
\notag d_{TV} \left( \frac{\widetilde{f}_n(Z)  - f_Z}{\sqrt{ \E [( \widetilde{f}_n(Z) - f_Z)^2] }}, \mathcal{N}\right) \leq  \varphi_{T_n}(Z) + C n \Delta_n^{2 \alpha + 1}.
\end{align}
Further bounds on $\varphi_{T_n}(Z)$ are presented, for instance, in Corollary~\ref{cor:explicit_bounds} when $|\rho(t)| = \mathcal{O}(t^{2 \beta -2})$ for some $\beta \in (0, 3/4)$.
\end{remark}

Finally, note that our approach for the proof of Theorem~\ref{CLT-tilde-STAT} can be applied directly to Corollary~\ref{cor:sigmaz} and then the following holds:
\begin{corollary} Let $\sigma_Z^2 = 4 \int_0^\infty \rho(u)^2du$. Under the same assumptions as in Theorem~\ref{CLT-tilde-STAT},   for all $n \geq 1$,
\begin{align}
\notag d_{TV}\left(\frac{\sqrt{T}}{\sigma_Z} (\widetilde{f}_n(Z) -
f_Z), \mathcal{N} \right) \leq \varphi_{T_n}(Z) + 2\left| 1-
\frac{\sigma_Z^2}{E(V_{T_n}(Z)^2)}\right| + C (n \Delta n^{2 \alpha
+ 1})^{1/4},
\end{align}
where $\varphi_{T_n}(Z)$ is as in Theorem~\ref{CLT-hat-STAT} and $C > 0$ is an absolute constant depending on $\E[Z_0^4]$. The same result holds for the Wasserstein distance.

\end{corollary}

\section{Parameter estimation for non-stationary Gaussian processes}\label{sec:asym}

In practical applications, the data rarely comes from a stationary process. This is reflected on the modeling side by for instance studying stochastic systems started at a point mass rather than the system's stationary distribution. Some more specific examples are included in Section~\ref{sec:OU}.

The present section is devoted to the general treatment of a process that is asymptotically stationary. In particular, let $Z$ be a centered stationary Gaussian process (as in Section~\ref{sec:stationary}), and let $Y$ be a stochastic process satisfying the following: there exists a constant $\gamma >1$ such that for every $p\geq 1$ and for all $T>0$,
\begin{align}
\left\Vert Y_T\right\Vert _{L^{p}}=\mathcal{O}\left( T^{-\gamma}\right) .  \label{hypothesis-on-Z+Y}
\end{align}
Then we consider second moment estimators for the process $X \coloneqq Z + Y$. Our goal is to estimate the limiting variance $f_X \coloneqq \lim_{T \to \infty} \E[ X_T^2]$. The limit exists and in fact $f_X = f_Z$. Indeed,
\begin{align}
\notag \E[ X_T^2 ] = \E[ (Z_T + Y_T)^2 ] = \E[ Z_T^2 ] + 2 \E [ Z_T Y_T] + \E[ Y_T^2] \to f_Z,
\end{align}
 since $ \E[ Z_T^2 ]  = f_Z$, $\E[ Y_T^2] = \mathcal{O}(T^{- 2 \gamma})$ and by the Cauchy-Schwarz inequality  $|\E [ Z_T Y_T] |  = \mathcal{O}(T^{-\gamma})$.
As in Section~\ref{sec:stationary} we consider the second moment estimators $\widehat{f}_T(X)$ and $\widetilde{f}_n(X)$, based on continuous-time and discrete-time observations of $X$:
\begin{align}
\label{estim-cont-NONSTAT} \widehat{f}_T(X) \coloneqq & \quad \frac{1}{T} \int_{0}^{T} X_{t}^{2}dt, \quad T>0, \\
\label{estim-disc-NONSTAT}\widetilde{f}_n(X) \coloneqq & \quad   \frac{1}{n} \sum_{i =1}^{n} X_{t_{i}}^{2}, \quad n \geq 1,
\end{align}
where $t_{i} = i \Delta_{n}$, $  i=0,\ldots, n $, $\Delta_{n}\rightarrow0$ and  $ T_{n} \coloneqq n \Delta_{n}\rightarrow \infty$.
We establish some basic properties of the two estimators.

\begin{proposition} The estimators $\widehat{f}_T(X)$ and $\widetilde{f}_n(X)$ are asymptotically unbiased. Moreover, there is a constant $C > 0$, such that for all $T> 0$,
\begin{align}
\label{eq:nonstat_unbiased} |\E [ \widehat{f}_T(X) - f_X]| \leq
CT^{- \gamma}, \mbox{ and } |\E [ \widetilde{f}_T(X) - f_X]| \leq
CT^{- \gamma}.
\end{align}
\end{proposition}
\begin{proof}
We established that for all $T > 0$, $\E[ X_T^2] = f_X + err(t)$, where the  error term satisfies $|err(t)| = \mathcal{O}(t^{- \gamma})$. Therefore, for the continuous estimator,
\begin{align}
\notag |\E[ \widehat{f}_T(X) - f_X ] | \leq \frac{1}{T} \int_0^T |err(t)| dt \leq C T^{ - \gamma} \to 0 \quad \mbox{as } T \to \infty,
\end{align}
and thus $\widehat{f}_T(X)$ is asymptotically unbiased. A similar computation yields the same for $\widetilde{f}_n(Z)$.

\end{proof}

We next show strong consistency.
\begin{proposition}\label{consistency for Z+Y}If $\int_{\R}\rho^2(r)dr<\infty$ and the condition (\ref{hypothesis-on-Z+Y})   holds, then
\begin{align}
\label{eq:bound_diff_asympt} \norm{\widehat{f}_T(X) - f_X}_{L^2} \leq CT^{-1/2}.
\end{align}
Moreover,
\begin{align}
\widehat{f}_{T}(X)\longrightarrow f_X \quad \mbox{almost surely as } T\to \infty.\label{CV-consistency-hat-X}
\end{align}
In addition, if  $Z$ satisfies the helix property (\ref{helix-property}), and $n\Delta_n^{2\alpha+1}\to 0$, as $n\rightarrow\infty$, then
\begin{align}
\widetilde{f}_{n}(X)\longrightarrow f_X \quad \mbox{almost surely as }  T\to \infty. \label{CV-consistency-discrete-X}
\end{align}
\end{proposition}

\begin{proof}
The first step, as before, is to establish a bound on $\norm{\widehat{f}_T(X) - f_X}_{L^2}$. We have that
\begin{align}
\notag \norm{\widehat{f}_T(X) - f_X}_{L^2} \leq & \quad \norm{\widehat{f}_T(X) - \widehat{f}_T(Z) }_{L^2} + \norm{ \widehat{f}_T(Z) - f_X}_{L^2}.
\end{align}
Note, that by the triangle inequality and the Cauchy-Schwarz inequality,
\begin{align}
\label{eq:estim_y} \norm{\widehat{f}_T(X) - \widehat{f}_T(Z) }_{L^2} \leq \sqrt{\E \left| \frac{1}{T} \int_0^T Y_t dt \right|^2 } \leq  \frac{1}{T} \left(\int_{[0, T]^2} \left( \E [ Y_t^2] \E [ Y_s^2] \right)^{1/2} ds dt \right)^{1/2}.
\end{align}
Therefore, using $\norm{Y_t}_{L^p} = \mathcal{O}(t^{- \gamma})$,
\begin{align}
\notag   \norm{\widehat{f}_T(X) - f_X}_{L^2} \leq C(T^{- \gamma} + T^{-1/2}) \leq CT^{-1/2},
\end{align}
where we have applied~\eqref{eq:VT2_finite}.

By hypercontractivity~\eqref{hypercontractivity} we can extend the bound to $L^p$ norms for $p \geq 1$. Then applying Lemma~\ref{Borel-Cantelli} and Lemma~\ref{continuous-Borel-Cantelli} consecutively establishes strong consistency for $\widehat{f}_T(X)$.

The proof for $\widetilde{f}_n(X)$ is similar and a main ingredient is the bound from~\eqref{estimate-difference}.
\end{proof}

We now turn to the alternative formulations for our main results Theorem~\ref{CLT-hat-STAT} and Theorem~\ref{CLT-tilde-STAT} when the stochastic process is asymptotically stationary. Recall that $\rho(t) = \rho(-t) \coloneqq \E[Z_0Z_t ]$ for $t \geq 0$.

\begin{theorem}\label{CLT-hat-NONSTAT}  Assume that $ \int_\R \rho^2(r) dr  < \infty$  and that the condition (\ref{hypothesis-on-Z+Y})   holds. Let   $\mathcal{N} \sim\mathcal{N}(0,1)$ be the standard normal random variable. Then, there exists a constant $C>0$ such that, for all $T>0$,
\begin{align}
d_{TV}\left(\frac{\widehat{f}_T(X) - f_X - \E[\widehat{f}_T(X) - f_X] }{\sqrt{ Var(\widehat{f}_T(X) - f_X)} }, \mathcal{N}\right) \leq \varphi_T(Z) + C T^{\frac{1-\gamma}{4}},
\end{align}
where $\varphi_T(Z)$ is defined in~\eqref{eq:varphi_bound}. The same result holds for the Wasserstein distance.
\end{theorem}

\begin{proof} By the triangle inequality and Lemma~\eqref{upper-d_TV-lemma}:
\begin{align}
\notag d_{TV} &\left(\frac{\widehat{f}_T(X) - f_X - \E[\widehat{f}_T(X) - f_X] }{\sqrt{ Var(\widehat{f}_T(X) - f_X)} }, \mathcal{N}\right)  \\
\notag \leq & \quad d_{TV}\left(\frac{\widehat{f}_T(Z) - f_Z  }{\sqrt{ Var(\widehat{f}_T(Z) - f_Z)} }, \mathcal{N}\right)  + d_{TV}\left(\frac{\widehat{f}_T(X) - f_Z }{\sqrt{ Var(\widehat{f}_T(Z) - f_Z)} }, \frac{\widehat{f}_T(Z) - f_Z }{\sqrt{ Var(\widehat{f}_T(Z) - f_Z)} }\right)  \\
\label{eq:nonstat_cont_breakdown} & \quad + \sqrt{\frac{\pi}{2}} \left| \frac{\E[ \widehat{f}_T(X) - f_X] }{\sqrt{Var(\widehat{f}_T(X) - f_X)}} \right| + 2 \left| 1 - \frac{Var(\widehat{f}_T(X) - f_X)}{Var(\widehat{f}_T(Z) - f_Z)}\right|.
\end{align}
We analyze each term on the right hand side of~\eqref{eq:nonstat_cont_breakdown}. First, by Theorem~\ref{CLT-hat-STAT},
\begin{align}
\notag d_{TV}\left(\frac{\widehat{f}_T(Z) - f_Z  }{\sqrt{ Var(\widehat{f}_T(Z) - f_Z)} }, \mathcal{N}\right) \leq \varphi_T(Z),
\end{align}
with $\varphi_T(Z)$ as defined in~\eqref{eq:varphi_bound}.

Then, using standard properties of the total variation distance and Lemma~\ref{dTV} :
\begin{align}
\notag d_{TV}&\left(\frac{\widehat{f}_T(X) - f_Z }{\sqrt{ Var(\widehat{f}_T(Z) - f_Z)} }, \frac{\widehat{f}_T(Z) - f_Z }{\sqrt{ Var(\widehat{f}_T(Z) - f_Z)} }\right)   \\
\notag =  & \quad d_{TV} \left( \widehat{f}_T(X) - f_Z, \widehat{f}_T(Z) - f_Z\right) \leq C \left( \frac{ \E |\widehat{f}_T(X) -  \widehat{f}_T(Z)|^2 }{\E | \widehat{f}_T(Z) - f_Z|^2 }\right)^{1/4} \leq C T^{(1 - \gamma)/4}.
\end{align}
Indeed, $\E |\widehat{f}_T(X) -  \widehat{f}_T(Z)|^2 = \mathcal{O}(T^{- \gamma})$ and $T \E | \widehat{f}_T(Z) - f_Z|^2 \to \sigma_Z^2 < \infty$, see~\eqref{eq:VT2_finite}. Moreover since $|\E[ \widehat{f}_T(X) - f_X ] | = \mathcal{O}(T^{- \gamma})$, one has
\begin{align}
\label{eq:diff_var} \left| Var(\widehat{f}_T(X)  - f_X) - Var(\widehat{f}_T(Z)  - f_Z) \right| = \mathcal{O}(T^{- \gamma}).
\end{align}
Then, there exists a constant $C > 0$ such that for all $T > 0$,
\begin{align}
\notag \sqrt{\frac{\pi}{2}} \left| \frac{\E[ \widehat{f}_T(X) - f_X] }{\sqrt{Var(\widehat{f}_T(X) - f_X)}} \right|  \leq  C T^{1/2 - \gamma} \quad \mbox{and}  \quad 2 \left| 1 - \frac{Var(\widehat{f}_T(X) - f_X)}{Var(\widehat{f}_T(Z) - f_Z)}\right| \leq  C T^{1 - \gamma}.
\end{align}
Therefore,
\begin{align}
\notag d_{TV} &\left(\frac{\widehat{f}_T(X) - f_X - \E[\widehat{f}_T(X) - f_X] }{\sqrt{ Var(\widehat{f}_T(X) - f_X)} }, \mathcal{N}\right) \leq \varphi_T(Z) + CT^{\frac{1 - \gamma}{4}},
\end{align}
as desired.
The corresponding bound for the Wasserstein distance follows from Remark~\ref{rem:wass_tv}.
\end{proof}

We now turn to the discrete version of the previous result.
\begin{theorem}\label{CLT-tilde-NONSTAT}  Assume $\int_{\R}\rho^2(r)dr<\infty$, $Z$ verifies (\ref{helix-property}), and the condition (\ref{hypothesis-on-Z+Y})   holds. Let $ \mathcal{N}\sim\mathcal{N}(0,1)$. Then, there is $C>0$ such that, for every $n\geq1$,
\begin{eqnarray*}
d_{TV}\left(\frac{\sqrt{T_n}}{\sigma_Z} (\widetilde{f}_n(X) - f_X), N\right) \leq \varphi_{T_n}(Z)+  C
\left[n\Delta_n^{2\alpha+1}\right]^{1/4}+CT_n^{\frac{1-\gamma}{4}}.
\end{eqnarray*}
\end{theorem}

\begin{proof} We follows the same steps as in the proof of Theorem~\ref{CLT-tilde-NONSTAT}. The extra term $Cn [ \Delta_n^{2 \alpha + 1}]^{1/4}$ comes from the bound in Theorem~\ref{CLT-tilde-STAT}.

\end{proof}

\section{Applications to Gaussian Ornstein-Uhlenbeck processes}\label{sec:OU}
In this section we  consider the Gaussian Ornstein-Uhlenbeck process $X^{\theta} \coloneqq \{X_{t}^{\theta}, t\geq 0\}$ defined via  the following linear stochastic differential equation
\begin{equation}
dX_{t}^{\theta}=-\theta X_{t}^{\theta}dt+dG_{t},\quad X_{0}^{\theta}=0,\quad \label{GOU}.
\end{equation}%
Here  $\theta>0$ is an unknown parameter and  $G$ is an arbitrary mean-zero Gaussian process such that $Z_{t}^{\theta} \coloneqq \int_{-\infty }^{t}e^{-\theta (t-s)}dG_s$, for $t \geq 0$,  is a stationary Gaussian process.
The equation~\eqref{GOU} has the following explicit solution (see~\cite{CKM})
\begin{equation*}
X_{t}^{\theta}=e^{-\theta t}\int_{0}^{t} e^{\theta s}dG_{s},\quad t\geq 0,
\end{equation*}
where the integral can be understood in the Wiener sense.  Then $X_t^{\theta} = Z_t^{\theta} + e^{-\theta t}Z_0^{\theta}$ thus  satisfies~\eqref{hypothesis-on-Z+Y} with $Y \coloneqq e^{- \theta t} Z_0^\theta$. Moreover, $\norm{Y_T}_{L^p} = \mathcal{O}(T^{- \gamma})$ is satisfied for every $p \geq 1$ and for every constant $\gamma > 1$.

As in the previous sections our goal is to establish a Berry-Esseen theorem involving estimators of the parameters of the process. Here we want to estimate  $\theta$. Our approach is based on writing
\begin{align}
\notag \theta = g_{Z^\theta} ( \E [ (Z_0^\theta)^2 ] ),
\end{align}
for some invertible function  $g_{Z^\theta} : \R^+ \to \R^+$. There are certain cases when there is an explicit expression for $g_{Z^\theta}$ (or its inverse). We explore those in the coming sections. For now, we describe the general treatment and make only one assumption for $g_{Z^\theta}$:
\begin{assumption}\label{assume:diffeo} The function $g_{Z^{\theta}}$ is a diffeomorphism, and is twice continuously differentiable.
\end{assumption}

We estimate $\theta $ based the continuous and discrete observations of $X$:
\begin{align}
\label{estim-cont-GOU} \widehat{\theta}_T \coloneqq & g_{Z^{\theta}} \left( \frac{1}{T} \int_{0}^{T}  (X_{t}^{\theta})^{2} dt \right) = g_{Z^{\theta}}\left(\widehat{f}_T(X^{\theta})\right),\quad  T>0, \\
\label{estim-disc-GOU} \widetilde{\theta}_n  \coloneqq & g_{Z^{\theta}}\left(\frac{1}{n} \sum_{i=1}^{n} (X_{t_{i}}^{\theta})^{2}\right)=g_{Z^{\theta}}\left(\widetilde{f}_n(X^{\theta})\right),\quad n\geq1,
\end{align}
where $t_{i} = i \Delta_{n}$, $  i=0,\ldots, n $, $\Delta_{n}\rightarrow0$ and  $ T_{n} = n \Delta_{n}\rightarrow \infty$, whereas $\widehat{f}_T(X^{\theta})$ and $\widetilde{f}_n(X^{\theta})$ are given by (\ref{estim-cont-NONSTAT}) and (\ref{estim-disc-NONSTAT}), respectively.

The following holds for the continuous estimator $\widehat{\theta}_T$.
 \begin{theorem}\label{CLT-hate-NONSTAT-FOU} Assume that $\int_{\R}\rho^2(r)dr<\infty$ and that Assumption~\ref{assume:diffeo} holds. Then for every  $\gamma >1$, and $p \geq 1$,
 \begin{align}
\label{eq:dw_bound_ou_gen} d_{W} \left( \frac{\widehat{\theta}_T - \theta - \E[ \widehat{\theta} - \theta] }{g'_{Z^\theta}(f_{X^\theta}) \sqrt{ Var( \widehat{f}_T(X^\theta) - f_{X^\theta})}} , \mathcal{N} \right) \leq C\frac{ \left(\E|g''(\zeta_{T})|^p\right)^{1/p}}{\sqrt{T}} +\varphi_T(Z^{\theta})+CT^{-1/2},
\end{align}
for some absolute constant $ C> 0$ and where $\zeta_T$ is a random variable in $[ \widehat{f}_T(X^\theta), f_{X^\theta}]$.
\end{theorem}

\begin{proof} Recall that by definition $\theta = g_{Z^\theta} ( f_{X^\theta})$. Under Assumption~\ref{assume:diffeo}
\begin{align}
\notag \left( \widehat{\theta}_{T}-\theta \right) =g_{Z^{\theta}}'(f_{X^{\theta}})\left(\widehat{f}_T (X^{\theta}) - f_{X^{\theta}}\right)+\frac{1}{2} g_{Z^{\theta}}''(\zeta_{T})\left(\widehat{f}_T(X^{\theta})-f_{X^{\theta}}\right)^2
\end{align}
for some random point $\zeta_{T}$ between $\widehat{f}_T(X^{\theta})$ and $f_{X^{\theta}}$. Denote $V_f^2 \coloneqq Var(\widehat{f}_T(X^\theta) - f_{X^\theta})$. Then,
\begin{align}
\notag d_{W}  \left(\frac{\widehat{\theta}_T-\theta - \E[ \widehat{\theta}_T -\theta]}{g_{Z^{\theta}}'(f_{X^{\theta}})V_f} , \mathcal{N} \right) \leq  & \frac{ |\E [\widehat{\theta} - \theta]| + \frac{1}{2} \E \left| (\widehat{f}_T(X^\theta) - f_{X^\theta})^2 g_{Z^\theta}'' (\zeta_T)  \right| }{  | g_{Z^\theta}' (f_{X^\theta} ) V_f |}  \\
\notag & + d_{W}\left(\frac{1}{V_f} (\widehat{f}_T(X^{\theta}) - f_{X^{\theta}}), \mathcal{N}\right),
\end{align}
where we have used that $d_W(x_1 + x_2, y) \leq  \E [| x_2|] + d_W(x_1, y)$ for any random variables $x_1, x_2, y$. This property of the Wasserstein distance is the main reason our results in Section~\ref{sec:OU} concern $d_W$ and not $d_{TV}$.

The second term in the  inequality above is bounded in Theorem~\ref{CLT-hat-NONSTAT}. By H\"older's inequality, and the hypercontractivity property, for  $p,q>1$ with $1/p + 1/q = 1$,
\begin{align}
\notag  \E \left| (\widehat{f}_T(X^\theta) - f_{X^\theta})^2 g_{Z^\theta}'' (\zeta_T)  \right| \leq  & \quad \left( \E|g''(\zeta_{T})|^p\right)^{1/p} \left( \E \left|\widehat{f}_T(X^\theta) - f_{X^\theta} \right|^{2q}  \right)^{1/q} \\
\notag \leq & \quad C\left( \E|g''(\zeta_{T})|^p\right)^{1/p} \E \left|\widehat{f}_T(X^\theta) - f_{X^\theta} \right|^2,
\end{align}
for some constant $C> 0$ depending on $p$. Moreover,
\begin{align}
\notag | \E [\widetilde{\theta} - \theta]| \leq |g_{Z^{\theta}}'(f_{X^{\theta}})| \left| \E [ \widehat{f}_T (X^{\theta}) - f_{X^{\theta}}] \right| + \E \left| (\widehat{f}_T(X^\theta) - f_{X^\theta})^2 g_{Z^\theta}'' (\zeta_T)  \right|.
\end{align}
Recall, that by~\eqref{eq:bound_diff_asympt}, $\E | \widehat{f}_T(X^\theta) - f_{X^\theta}|^2 \leq C /T$, and by~\eqref{eq:nonstat_unbiased} $|\E [ \widehat{f}_T(X^\theta) - f_{X^\theta}] | \leq C T^{- \gamma}$.
Therefore,
\begin{align}
\notag \E | \widehat{\theta} - \theta| + \frac{1}{2} \E \left| (\widehat{f}_T(X^\theta) - f_{X^\theta})^2 g_{Z^\theta}'' (\zeta_T)  \right| \leq C \left( |g'_{Z^\theta}(f_{X^\theta}) | + \frac{3}{2} \left(\E|g''(\zeta_{T})|^p\right)^{1/p} \right) T^{- 1}.
\end{align}
 Now, by~\eqref{eq:VT2_finite} and~\eqref{eq:diff_var}, $ TV_f^2  \to \sigma_z^2$. The bound~\eqref{eq:dw_bound_ou_gen} follows.

\end{proof}

Similarly, we obtain the rate of convergence in law of $ \sqrt{T_n} (\widetilde{\theta}_n-\theta)$ as follows.

 \begin{theorem}\label{CLT-tilde-NONSTAT-FOU} Suppose that the conditions of Theorem  \ref{CLT-tilde-NONSTAT} hold. If $g''(\xi_{n})$, with $\xi_{n} \in [|\widetilde{f}_n(X^{\theta}), f_{X^{\theta}}|]$, has a moment of order greater than $1$ which is bounded in $n$,
 \begin{eqnarray*}
d_{W}\left(\frac{\sqrt{T_n}}{\sigma_{Z^{\theta}}g_{Z^{\theta}}'(f_{X^{\theta}})}(\widetilde{\theta}_n-\theta), \mathcal{N} \right) \leq \frac{C}{\sqrt{T_n}} +\varphi_{T_n}(Z^{\theta})+  C\left[n\Delta_n^{2\alpha+1}\right]^{1/4}.
\end{eqnarray*}
\end{theorem}

\subsection{Fractional Ornstein-Uhlenbeck process of the first kind}\label{FOUsect}

Here we consider the Ornstein-Uhlenbeck process $X^{\theta} \coloneqq \left\{X^{\theta}_{t},t\geq 0\right\} $ driven by a fractional Brownian motion $\left\{B_{t}^{H},t\geq 0\right\} $ of Hurst index $H\in (0,1)$.  More precisely,  $X^{\theta}$ is the solution of the following linear stochastic differential equation
\begin{equation}
X^{\theta}_{0}=0;\quad dX^{\theta}_{t}=-\theta X^{\theta}_{t}dt+dB_{t}^{H},\quad t\geq 0, \label{FOU}
\end{equation}
where $\theta >0$ is an unknown parameter. The process $X^\theta$ is called a fractional Ornstein-Uhlenbeck process of the first kind, following the notation in~\cite{KS}. As in the general case described above there is an explicit solution to~\eqref{FOU}:
\begin{align}
X^{\theta}_{t}=\int_{0}^{t}e^{-\theta (t-s)}dB_{s}^{H}. \label{fOUX}
\end{align}
Moreover,
\begin{equation}
Z_{t}^{\theta }=\int_{-\infty }^{t}e^{-\theta (t-s)}dB_{s}^{H}.\label{Ztheta}
\end{equation}
 is a stationary Gaussian process, see \cite{CKM,EV}. The process $Z^\theta$ is also the stationary solution of equation (\ref{FOU}) when $X^{\theta}_0=Z_{0}^{\theta }$. We are thus in the setup of Section~\ref{sec:asym} with $X^\theta = Z^\theta + Y^\theta$, where $Y_t^\theta=-e^{-\theta t}Z_{0}^{\theta }$. Again, $\norm{Y_t}_{L^p} = \mathcal{O}(t^{- \gamma})$ for every $p \geq 1$ and any $\gamma > 1$.
Note that by integration by parts and the formula for covariance for fractional Brownian motion
\begin{align}
\notag \E [(Z_0^\theta)^2] = &  \E \left( \int_{-\infty}^0 e^{\theta s} dB_s^H \right)^2 = \E \int_{-\infty}^0 \int_{-\infty}^0 B_s^H B_r^H e^{\theta(s + r) }ds dr \\
\notag  = & \theta^2  \int_{-\infty}^0 \int_{-\infty}^0  \frac{1}{2}e^{\theta(s + r)} (|r|^{2H} + |s|^{2H}- |r-s|^{2H} ) dr ds \\
\notag = &  \theta^{-2H} H \Gamma(2H).
\end{align}
Therefore,  $\theta = g_{Z^\theta} (\E[ (Z_0^\theta)^2]$ where the invertible function $g_{Z^\theta}: \R^+ \to \R^+$ is given by
\begin{align}
\label{eq:first_gztheta}g_{Z^{\theta}}(x) =\left(\frac{H\Gamma
(2H)}{x}\right)^{\frac{1}{2H}},\quad x>0.
\end{align}
The estimators $\widehat{\theta}_T$ and $\widetilde{\theta}_T$ given by~\eqref{estim-cont-GOU} and~\eqref{estim-disc-GOU} were carefully studied in~\cite{HNZ}. In particular, due to~\cite[Theorem 9]{HNZ} and~\cite[Theorem 11]{HNZ} the following holds:
\begin{proposition} Let $H \in (0, 3/4)$. The estimators $\widehat{\theta}_T$ and $\widetilde{\theta}_n$ are strongly consistent. Denote
\begin{align}
\label{variance-fOU} \delta_H^2 \coloneqq  \frac{\theta}{(2H)^2} \times \left\{ \begin{array}{ll} (4 H-1)+\frac{2 \Gamma(2-4 H) \Gamma(4 H)}{\Gamma(2 H) \Gamma(1-2 H)} & \text { if } H \in\left(0, \frac{1}{2}\right), \\
(4 H-1) \left( 1+\frac{\Gamma(3-4 H) \Gamma(4 H-1)}{\Gamma(2-2 H) \Gamma(2 H)}\right) & \text { if } H \in\left[ \frac{1}{2}, \frac{3}{4} \right). \end{array} \right.
\end{align}
The following limit theorems hold:
\begin{enumerate}
\item $\sqrt{T}(\widehat{\theta}_T - \theta) \overset{\mathcal{L}}{\rightarrow} \mathcal{N}(0, \delta_H^2 ) $ as $T \to \infty$.

\item Assume there is $p \in (1, \frac{3 + 2H}{1 + 2H} \wedge (1 + 2H))$ such that $n \Delta_n^p \to 0$. Then $\sqrt{T_n} (\widetilde{\theta} - \theta)  \overset{\mathcal{L}}{\rightarrow} \mathcal{N}(0, \delta_H^2 ) $ as $n \to \infty$.
\end{enumerate}

\end{proposition}

We can extend these results and obtain Berry-Es\'een bounds using Theorems~\ref{CLT-hate-NONSTAT-FOU} and~\ref{CLT-tilde-NONSTAT-FOU}.
\begin{theorem}\label{thm:explicit_bounds_ou1} Let $H\in(0,\frac34)$, and  $\delta_H$ be given by~\eqref{variance-fOU}. Then
\begin{align}
\label{eq:bound_hat_1}d_{W}\left(\frac{\sqrt{T}}{\delta_H} (\widehat{\theta}_T-\theta), \mathcal{N} \right)  & \leq   C \left\{ \begin{array}{ll} \frac{1}{\sqrt{T}} & \mbox{ if }0<H \leq\frac{5}{8},  \\
~~ &  \\
\frac{1}{T^{3-4H}} & \mbox{ if }\frac{5}{8}<H<\frac34. \end{array}\right.
\end{align}
Moreover,
\begin{align}
\label{eq:bound_tilde_1}d_{W}\left( \frac{\sqrt{T_n}}{\delta_H}( \widetilde{\theta}_n - \theta ), \mathcal{N} \right)  & \leq  C \left[ n\Delta_n^{2H+1} \right]^{1/2}+  C \left\{ \begin{array}{ll} \frac{1}{ \sqrt{n \Delta_n} } & \mbox{ if } 0<H \leq \frac{5}{8},  \\
~~ &  \\
\frac{1}{(n\Delta_n){3-4H}} & \mbox{ if }\frac{5}{8}<H<\frac{3}{4}. \end{array}\right.
\end{align}

\end{theorem}
\begin{proof}

Recall that $\E[Z_0 Z_r] = \mathcal{O}(r^{2H -2})$ and $H \in (0, 3/4)$, see~\cite[Proposition 4.2]{DEV}. Thus we are in the setting of Corollary~\ref{cor:sigmaz}. Then $\varphi_T(Z)$ and $\varphi_{T_n}(Z)$ are respectively bounded via~\eqref{eq:explicit_sz} with $\beta = H$.

To establish~\eqref{eq:bound_hat_1} and~\eqref{eq:bound_tilde_1} it is left to show that $\E |g''(\zeta_T)|^p < \infty$ for some $p \geq 1$. Using the monotonocity of $g''$ and the fact that $\zeta_T \in [|\widehat{f}_T(X^{\theta}), f_{X^{\theta}}|]$, it is enough to show that $\E |\widehat{f}_T(X^{\theta})|^p < \infty$ for some $p \geq 1$. This follows as an application of the technical Proposition~\ref{prop:bound_ft_p} in  Section~\ref{sec:bound_ft_p}. Indeed, since $X_t^\theta = Z_t^\theta - e^{\theta t}Z_0^\theta$ for a centered stationary Gaussian process $(Z_t^\theta)_{t \geq 0}$,  one needs only to check that $Z^\theta$ satisfies the condition~\eqref{cond-karhunen}, i.e,
\begin{align}
\label{eq-check}  \lim_{t\rightarrow0}  \E[(Z_t^\theta - Z_0^\theta)^2] = 0,\quad \mbox{ and }   \bigcap_{t\in\R}\overline{sp}\{Z_s^\theta:-\infty<s\leq t\} = \{0\},
\end{align}
where $\overline{sp}$ denotes the $L^2$-closure of the linear span of a set of square-integrable random variables. Note, that by~\eqref{Ztheta}, and integration by parts,
\begin{align}
\notag \E[(Z_t^\theta - Z_0^\theta)^2]  = \E \int_0^t \int_0^t B_s^H B_r^H e^{\theta(s +r)} ds dr =  \int_0^t \int_0^t \frac{\theta^2}{2} e^{\theta (s + r) } (r^{2H} + s^{2H} - |r - s|^{2H} ds dr,
\end{align}
which in turn approaches $0$ as $t \to 0$ and the first part of~\eqref{eq-check} is satisfied. For the second part, note that $\{Z_s^\theta: - \infty < s \leq t \} \subset \sigma( B_s^H : - \infty < s \leq t)$ and the intersection of the sigma algebras is empty.
\end{proof}

\subsection{ Fractional Ornstein-Uhlenbeck process of the second kind\label{FOUSKsection}}

The last example we consider is the so-called fractional Ornstein-Uhlenbeck process of the second kind, defined via the stochastic differential equation
\begin{equation}
S_{0}^{\mu }=0,\mbox{ and }\ dS_{t}^{\mu }=-\mu S_{t}^{\mu }dt+dY_{t}^{(1)},\quad t\geq 0, \label{FOUSK}
\end{equation}%
where $Y_{t}^{(1)}=\int_{0}^{t}e^{-s}dB^H_{a_{s}}$ with $a_{s}=He^{\frac{s}{H}} $ and $\left\{ B^H_{t},t\geq 0\right\} $ is a fractional Brownian motion with Hurst parameter $H\in \left(\frac{1}{2},1\right)$,  and where $\mu >0$ is the unknown real parameter which we would like to estimate. The equation (\ref{FOUSK}) admits an explicit solution, see~\cite[Equation 3.9]{KS}:
\begin{equation*}
S_{t}^{\mu }=e^{-\mu t}\int_{0}^{t}e^{\mu s}dY_{s}^{(1)}=e^{-\mu t}\int_{0}^{t}e^{(\mu -1)s}dB^H_{a_{s}}=H^{(1-\mu )H}e^{-\mu t}\int_{a_{0}}^{a_{t}}r^{(\mu -1)H}dB^H_{r}.
\end{equation*}
Hence we can also write
\begin{equation*}
S_{t}^{\mu }=Z_{t}^{\mu }-e^{-\mu t}Z_{0}^{\mu },
\end{equation*}
where
\begin{equation*}
Z_{t}^{\mu }=e^{-\mu t}\int_{-\infty }^{t}e^{(\mu -1)s}dB^H_{a_{s}}=H^{(1-\mu )H}e^{-\mu t}\int_{0}^{a_{t}}r^{(\mu -1)H}d B^H_{r}.
\end{equation*}
From~\cite[Lemma 37]{EV}, for every $H\in (\frac{1}{2},1)$,
\begin{equation*}
g_{Z^{\mu}}^{-1}(\mu) =f_{X^{\mu}}=f_{Z^{\mu}}=E\left[ \left( Z_{0}^{\mu }\right) ^{2}\right] =\frac{(2H-1)H^{2H}}{\mu} \mathcal{B} (1-H+\mu H,2H-1),
\end{equation*}%
where here $\mathcal{B}(\cdot)$ is the usual beta function. The function $\mu \mapsto g_{Z^{\mu}}^{-1}(\mu)$ is monotone (decreasing) and convex from $\R_{+}$ to $\R_{+}$. Now, the following Berry-Esseen result holds:
\begin{theorem} Assume $H\in(\frac12,1)$. Then
\begin{eqnarray*}
d_{W}\left(\frac{\sqrt{T}}{\sigma_{Z^{\mu}}g_{Z^{\mu}}'(f_{X^{\mu}})} (\widehat{\mu}_T-\mu), \mathcal{N}\right)  & \leq   C/\sqrt{T}.
\end{eqnarray*}
Also,
\begin{eqnarray*}
d_{W}\left(\frac{\sqrt{T_n}}{\sigma_{Z^{\mu}}g_{Z^{\mu}}'(f_{X^{\mu}})} (\widetilde{\mu}_n-\mu), \mathcal{N}\right)  & \leq  C \left[n\Delta_n^{2H+1}\right]^{1/2}+  C /\sqrt{n\Delta_n}.
\end{eqnarray*}
\end{theorem}

\begin{proof} From~\cite{KS}, there exist $c,C>0$ such that for all large $|t|$,
\begin{equation*}
\rho_{Z^{\mu }}(t)= \E\left[ Z_{0}^{\mu }Z_{t}^{\mu }\right] \leq Ce^{-c|t|}.
\end{equation*}
Thus, by a straightforward calculation,
\begin{eqnarray*}
 \varphi_T(Z^{\mu}) & \leq   C/\sqrt{T}.
\end{eqnarray*}
Moreover, according to~\cite[Lemma 4]{BEV}, for all $H \in(0,1)$, and $|t-s|$ small enough, $\E\left|Y_{t}^{(1)}-Y_{s}^{(1)}\right|^{2}=C|t-s|^{2H}$ and we are in the setting of Corollary~\ref{cor:sigmaz}    .

Finally, one needs to bound $\E |g''(\zeta_T)|^p$. As in the proof of Theorem~\ref{thm:explicit_bounds_ou1} we use the technical proposition Proposition~\ref{prop:bound_ft_p} in  Section~\ref{sec:bound_ft_p}. Indeed, since $S_t^\mu = Z_t^\mu - e^{-\mu t}Z_0^\mu$ for a centered stationary Gaussian process $(Z_t^\mu)_{t \geq 0}$,  one needs only to check that $Z^\mu$ satisfies the condition~\eqref{cond-karhunen}, i.e,
\begin{align}
\label{eq-check}  \lim_{t\rightarrow0}  \E[(Z_t^\mu - Z_0^\mu)^2] = 0,\quad \mbox{ and }   \bigcap_{t\in\R}\overline{sp}\{Z_s^\mu:-\infty<s\leq t\} = \{0\},
\end{align}
where $\overline{sp}$ denotes the $L^2$-closure of the linear span of a set of square-integrable random variables. Recall~\cite[Lemma 2.1]{AV}. One has $\E[(Z_t^\mu - Z_0^\mu)^2]  \sim Ht^{2H}$ as $t \to 0^+$. Thus the first part of~\eqref{eq-check} is satisfied. For the second part, note that $\{Z_s^\theta: - \infty < s \leq t \} \subset \sigma( B_{He^{s/H}}^H : - \infty < s \leq t)$ and the intersection of the sigma algebras is empty.

\end{proof}

\section{Technical results}\label{sec:technical}

\subsection{Two technical lemmas}
\begin{proof}[Proof of Lemma~\ref{continuous-Borel-Cantelli} ]  We employ similar arguments to~\cite[Proposition 4.1]{NT}. First, write
\[\frac{1}{T}\int_0^T u_t dt=\frac{1}{\lfloor{T}\rfloor}\int_0^{\lfloor{T}\rfloor} u_t dt +\frac{1}{T}\int_{\lfloor{T}\rfloor}^T u_t dt+\left(\frac{1}{T}-\frac{1}{\lfloor{T}\rfloor}\right)\int_0^{\lfloor{T}\rfloor} u_t dt.\]
Notice that
\begin{align}
\notag\left|\left(\frac{1}{T}-\frac{1}{\lfloor{T}\rfloor}\right)\int_0^{\lfloor{T}\rfloor} u_t dt\right| = \left(1-\frac{\lfloor{T}\rfloor}{T} \right) \left|\frac{1}{\lfloor{T}\rfloor}\int_0^{\lfloor{T}\rfloor} u_t  dt\right| \leq  \left|\frac{1}{\lfloor{T}\rfloor}\int_0^{\lfloor{T}\rfloor} u_t dt\right|
\end{align}
We know that $\frac{1}{\lfloor T \rfloor} \int_0^{\lfloor T \rfloor } u_t dt \to 0$ almost surely as $T \to 0$. Thus, we are left to show that
\begin{align}
\label{eq:fractions_almsure} \left|\frac{1}{T}\int_{\lfloor{T}\rfloor}^T u_t dt\right| \to 0, \mbox{ almost surely as } T \to \infty.
\end{align}
Recall, that $\sup_{t \geq 0} \E |u_t|^p < C_p$ for any $p \geq 1$. Thus, by Minkowski's inequality
\begin{align}
\notag \left(\E\left| \frac{1}{\lfloor{T}\rfloor}\int_{\lfloor{T}\rfloor}^{\lfloor{T}\rfloor+1} \left|u_t\right| dt\right|^p\right)^{1/p}\leq \frac{1}{\lfloor T \rfloor} \int_{\lfloor{T}\rfloor}^{\lfloor{T}\rfloor+1} \left( \E |u_t|^p \right)^{1/p} dt  \leq  \frac{C_p^{1/p}}{\lfloor{T}\rfloor}.
\end{align}
Now, by Lemma~\ref{Borel-Cantelli},
\begin{align}
\notag \left|\frac{1}{\lfloor T \rfloor}\int_{\lfloor{T}\rfloor}^{\lfloor{T}\rfloor +  1} |u_t| dt\right| \to 0, \mbox{ almost surely as } T \to \infty,
\end{align}
and~\eqref{eq:fractions_almsure} follows.
\end{proof}
\medskip

\begin{proof}[Proof of Lemma~\ref{rate-cumulants}]
 According to~\eqref{chaos-V} and since $\E[V_T(Z) ] =0$,  we have
\begin{align}
\notag \kappa_3(V_T(Z)) = & \quad \E [ V_T(Z)^3] = \E \left[ \left( \frac{1}{\sqrt{T}} \int_0^T I_2(\stepid{t}^{\otimes 2}) \right)^3 \right] \\
\notag = & \quad \frac{1}{T^{3/2}} \int_{[0,T]^3} \E \left[ I_2(\stepid{r}^{\otimes 2}) I_2(\stepid{s}^{\otimes 2}) I_2(\stepid{t}^{\otimes 2})\right] dr ds dt.
\end{align}
Applying the product formula~\eqref{eq:product}
\begin{align}
\label{eq:prod2} I_2(\stepid{s}^{\otimes 2}) I_2(\stepid{t}^{\otimes 2}) = I_4(\stepid{s}^{\otimes 2} \widetilde{\otimes}_0 \stepid{t}^{\otimes 2}) + 4 I_2(\stepid{s}^{\otimes 2} \widetilde{\otimes}_1 \stepid{t}^{\otimes 2})  + I_0(\stepid{s}^{\otimes 2} \widetilde{\otimes}_2 \stepid{t}^{\otimes 2}).
\end{align}
Now, by isometry~\eqref{eq:isometry}
\begin{align}
\notag \E \left[ I_2(\stepid{r}^{\otimes 2}) I_2(\stepid{s}^{\otimes 2}) I_2(\stepid{t}^{\otimes 2})\right] = 4\E \left[ I_2(\stepid{r}^{\otimes 2}) I_2(\stepid{s}^{\otimes 2} \widetilde{\otimes}_1 \stepid{t}^{\otimes 2}) \right] = 8 \inner{\stepid{r}^{\otimes 2}}{(\stepid{s}^{\otimes 2} \widetilde{\otimes}_1 \stepid{t}^{\otimes 2}}_{\HC^{\otimes 2}}.
\end{align}
Recall the formula for contraction~\eqref{eq:contraction}. Then,
\begin{align}
\notag \stepid{s}^{\otimes 2} \otimes_1 \stepid{t}^{\otimes 2} =\inner{\stepid{s}}{\stepid{t}}_\HC \stepid{s} \otimes \stepid{t}.
\end{align}
The symmetrization is then
\begin{align}
\notag \stepid{s}^{\otimes 2} \widetilde{\otimes}_1 \stepid{t}^{\otimes 2} = \frac{1}{2} \inner{\stepid{s}}{\stepid{t}}_\HC \left( \stepid{s} \otimes \stepid{t} + \stepid{t} \otimes \stepid{s}\right).
\end{align}
Therefore,
\begin{align}
\notag \E \left[ I_2(\stepid{r}^{\otimes 2}) I_2(\stepid{s}^{\otimes 2}) I_2(\stepid{t}^{\otimes 2})\right]  = 8  \inner{\stepid{r}}{\stepid{s}}_\HC  \inner{\stepid{r}}{\stepid{t}}_\HC  \inner{\stepid{s}}{\stepid{t}}_\HC,
\end{align}
and then
\begin{align}
\notag \kappa_3(V_T(Z)) = & \quad \frac{8}{T^{3/2} } \int_{[0, T]^3} \inner{\stepid{r}}{\stepid{s}}_\HC  \inner{\stepid{r}}{\stepid{t}}_\HC  \inner{\stepid{s}}{\stepid{t}}_\HC dr ds dt \\
\notag = & \quad \frac{8}{T^{3/2} } \int_{[0, T]^3} \E[ Z_r Z_s] \E [Z_r Z_t] \E [ Z_s Z_t] dr ds dt \\
\notag = &  \quad \frac{8}{T^{3/2}} \int_{[0,T]^3} \rho(r-s) \rho(r-t) \rho(s-t) dr ds dt \\
\notag =& \quad \frac{8}{T^{3/2}} \int_{0}^{T} \int_{-t}^{T-t} \int_{-t}^{T-t} \rho(x-y)\rho(x)\rho(y) dx dy dt \\
\notag = & \quad \frac{8}{T^{3/2}} \int_{-T}^{T} \int_{-T}^{T} \rho(x-y) \rho(x) \rho(y) dx dy \int_{(-x)\vee(-y)\vee0}^{(T-x)\wedge(T-y)\wedge T}dt
\end{align}
Let $\rho_T(x) \coloneqq |\rho(x) | \mathbf{1}_{|x| \leq T}$. Then,
\begin{align}
\notag \kappa_3(V_T(Z)) \leq & \quad \frac{8}{\sqrt{T}} \int_\R \int_\R \rho_T(x-y) \rho(x) \rho(y) dx dy \\
\notag = & \quad \frac{8}{\sqrt{T}} \int_\R (\rho_T * \rho_T) (y) \rho_T(y) dy \\
\notag \leq &\quad  \norm{\rho_T * \rho_T}_{L^3(\R)} \norm{\rho_T}_{L^{3/2}(\R)},
\end{align}
where $\rho_T * \rho_T$ is the convolution of the two functions and we have applied H\"older's inequality in the last line. Now, recall Young's inequality: If $p, q, r \geq 1$, $f \in L^p(\R)$, $g \in L^q(\R)$ and $1/p + 1/q = 1/r + 1$, then
\begin{align}
\notag \norm{f * g}_{L^r(\R)} \leq \norm{f}_{L^p(\R)} \norm{g}_{L^q(\R)}.
\end{align}
 Hence,
\begin{align}
\notag \kappa_3(V_T(Z)) \leq \norm{\rho_T}_{L^{3/2}(\R)}^3 = \left( \int_{-T}^T |\rho(T)|^{3/2} \right)^2,
\end{align}
and~\eqref{third-cumulant} is established.

Similarly, we have
\begin{align}
\notag \kappa_4(V_T(Z))  = &  \E [ V_T(Z)^4]  - 3 \E[V_T(Z)^2]^2 = \frac{1}{T^2}\E \left[ \left( \int_0^T I_2(\stepid{t}^{\otimes 2}) \right)^4 \right] - \frac{3}{T^2} \E \left[ \left(  \int_0^T I_2(\stepid{t}^{\otimes 2}) \right)^2 \right]^2\\
\notag = & \quad \frac{1}{T^2} \int_{[0,T]^4} \E \left[  I_2(\stepid{s}^{\otimes 2}) I_2(\stepid{t}^{\otimes 2})I_2(\stepid{u}^{\otimes 2}) I_2(\stepid{v}^{\otimes 2})\right] ds dt du dv \\
\notag & - \frac{3}{T^2} \left( \int_{[0, T]^2} \E \left[
I_2(\stepid{t}^{\otimes 2}) I_2(\stepid{s}^{\otimes 2})\right] dt ds
\right)^2.
\end{align}
By~\eqref{eq:prod2} and the isometry property
\begin{align}
\notag \E  & \left[ I_2(\stepid{s}^{\otimes 2}) I_2(\stepid{t}^{\otimes 2})I_2(\stepid{s}^{\otimes 2}) I_2(\stepid{t}^{\otimes 2})\right]  \\
\notag = \quad &   \E \left[  I_4(\stepid{s}^{\otimes 2} \widetilde{\otimes}_0 \stepid{t}^{\otimes 2}) I_4(\stepid{u}^{\otimes 2} \widetilde{\otimes}_0 \stepid{v}^{\otimes 2}) \right]   + 16 \E \left[ I_2(\stepid{s}^{\otimes 2} \widetilde{\otimes}_1 \stepid{t}^{\otimes 2})  I_2(\stepid{u}^{\otimes 2} \widetilde{\otimes}_1 \stepid{v}^{\otimes 2}) \right]  \\
\notag  &  + \E \left[  I_0(\stepid{s}^{\otimes 2} \widetilde{\otimes}_2 \stepid{t}^{\otimes 2})
  I_0(\stepid{u}^{\otimes 2} \widetilde{\otimes}_2 \stepid{v}^{\otimes 2})
  \right].
\end{align}
For the contractions the following holds:
\begin{align}
\notag \stepid{s}^{\otimes 2} \otimes_0 \stepid{t}^{\otimes 2} =&  \stepid{s} \otimes  \stepid{s} \otimes \stepid{t} \otimes \stepid{t},\\
\notag \stepid{s}^{\otimes 2} \otimes_1 \stepid{t}^{\otimes 2} =& \inner{\stepid{s}}{\stepid{t}}_\HC \stepid{s} \otimes \stepid{t},\\
\notag \stepid{s}^{\otimes 2} \otimes_2 \stepid{t}^{\otimes 2} =& \inner{\stepid{s}}{\stepid{t}}_\HC \inner{\stepid{s}}{\stepid{t}}_\HC.
\end{align}
After taking symmetrization into account and by the isometry formula~\eqref{eq:isometry},
\begin{align}
\notag \E \left[  I_4(\stepid{s}^{\otimes 2} \widetilde{\otimes}_0 \stepid{t}^{\otimes 2}) I_4(\stepid{u}^{\otimes 2} \widetilde{\otimes}_0 \stepid{v}^{\otimes 2}) \right]  =  &  4 \left(\E[Z_s Z_u] \E [Z_t Z_v]\right)^2  + 4 \left(\E[Z_s Z_u] \E [Z_t Z_v]\right)^2\\
\notag & + 16 \E[Z_s Z_u] \E[Z_s Z_v] \E[Z_t Z_u] \E [Z_t Z_v], \\
\notag \E \left[ I_2(\stepid{s}^{\otimes 2} \widetilde{\otimes}_1 \stepid{t}^{\otimes 2})  I_2(\stepid{u}^{\otimes 2} \widetilde{\otimes}_1 \stepid{v}^{\otimes 2}) \right]  = &  \E[ Z_s Z_t] \E[Z_u Z_v] \E[Z_s Z_u] \E[Z_t Z_v] \\
\notag & + \E [Z_s Z_t] \E [Z_u Z_v] \E [Z_s Z_v] \E[Z_t Z_u], \\
\notag \E\left[  I_0(\stepid{s}^{\otimes 2} \widetilde{\otimes}_2 \stepid{t}^{\otimes 2})  I_0(\stepid{u}^{\otimes 2} \widetilde{\otimes}_2 \stepid{v}^{\otimes 2}) \right] = & (\E[Z_s Z_t] \E[Z_u Z_v])^2.
\end{align}

Therefore by symmetry and~\eqref{eq:second_exp},
\begin{align}
\notag \kappa_4(V_T(Z)) = & \quad \frac{48}{T^2}  \int_{[0, T]^4} \E[ Z_s Z_u] \E [Z_s Z_v] \E[ Z_t Z_u] \E [Z_t Z_v]  ds dt du dv \\
\notag & + \frac{4 + 4 + 1 - 3*4}{T^2} \int_{[0, T]^4} \E[ Z_s Z_t]^2 \E [Z_u Z_v]^2  ds dt du dv \\
\notag \leq  &\quad  \frac{48}{T^2}  \int_{[0, T]^4}   \rho(s-u) \rho(s-v) \rho(t-u) \rho(t-v) ds dt du dv \\
\notag = & \quad \frac{48}{T^2}  \int_{[0, T]^2} \int_{\R^2} \rho_T(u-s) \rho_T(v-s) \rho_T(t-u) \rho_T(t-v) du dv  ds dt \\
\notag = & \quad  \frac{48}{T^2}  \int_{[0, T]^2} \int_{\R^2} \rho_T (t-s - x) \rho_T(t-s-y) \rho_T(x) \rho_T(y) dx dy ds dt\\
\notag = & \quad \frac{48}{T^2}  \int_{[0, T]^2} ((\rho_T * \rho_T)(t-s) )^2 ds dt\\
\notag = & \quad \frac{48}{T^2} \norm{\rho_T * \rho_T}_{L^2(\R)}^2,
\end{align}
where we have also applied the change of variables $x = t-u$, $y = t-v$ and used the formula for convolution. Next, by  Young's inequality,
\begin{align}
\notag \norm{\rho_T * \rho_T}_{L^2(\R)}^2 \leq  \norm{\rho_T}_{L^{4/3}(\R)}^3,
\end{align}
and this establishes~\eqref{fourth-cumulant}

\end{proof}

\subsection{A bound on $\E|\widehat{f}_T(X)|^p$}\label{sec:bound_ft_p}

First, we recall an important representation for stationary Gaussian processes.
\begin{definition}
Let $\phi_S$ denote the set of functions $\xi\in L^2(\R)$ such that $\xi(t) = 0$ for all $t <0$.  If $\xi \in \phi_S$, we can define for all $t \in \R$,
\begin{align}
\label{rep-integ} Z_t^\xi \coloneqq \int_\R \xi(t - u) dW_u,
\end{align}
where $(W_t)_{t \geq 0}$ is the Wiener process. The process $(Z_t^\xi)_{t \in \R}$ is a stationary centered Gaussian process.
\end{definition}

The following result is key for our approach.
\begin{theorem}[Karhunen \cite{karhunen}]\label{thm-karhunen}
Let $\{Z_t,t\in\R\}$ be a stationary centered Gaussian process such that
\begin{align}
 \lim_{t\rightarrow0}  \E[(Z_t - Z_0)^2] = 0,\quad  \bigcap_{t\in\R}\overline{sp}\{Z_s:-\infty<s\leq t\} = \{0\},\label{cond-karhunen}
 \end{align}
where $\overline{sp}$ denotes the $L^2$-closure of the linear span of a set of square-integrable random variables.

Then  there exists a $\xi\in \phi_S$  such that
\begin{align}
\notag \{ Z_t, t \in \R \} \overset{(d)}{=} \{ Z_t^{\xi}, t \in \R \},
 \end{align}
where $(d)$ denotes equality of all finite-dimensional distributions, and   $\{Z_t^{\xi}, t \in \R \}$ is a stationary centered Gaussian process defined as in~\eqref{rep-integ}.
\end{theorem}

Now, our goal is to show the following:

\begin{proposition}\label{prop:bound_ft_p}
Let $\{Z_t, t\geq0\}$ be a stationary centered Gaussian process  satisfying~\eqref{cond-karhunen}, and $\E[Z_0^2]=1$.  Then, for every $p>0$, there is $T_0>0$ such that
\begin{align}
\label{average-cont} \sup_{T\geq T_0} \E\left[\left(\frac{1}{T}\int_0^T
 Z_t^2dt\right)^{-p}\right]<\infty, \mbox{ and } \sup_{T\geq T_0} \E\left[\left(\frac{1}{T}\int_0^T(Z_t-e^{\theta t
 }Z_0)^2dt\right)^{-p}\right]<\infty.
\end{align}
Moreover,  for every $p>0$, there is $n_0 \geq 1$ such that
\begin{eqnarray}
\sup_{n\geq n_0} \E\left[\left(\frac{1}{n} \sum_{i =1}^{n} Z_{t_{i}}^{2}\right)^{-p}\right]<\infty, \mbox{ and } \sup_{n\geq n_0} \E\left[\left(\frac{1}{n} \sum_{i =1}^{n} (Z_{t_{i}}-e^{\theta t_{i} }Z_0)^{2}\right)^{-p}\right]<\infty, \label{average-disc}
\end{eqnarray}
where $t_i = i \Delta_n$, for $i = 1, \ldots, n$, with $n \Delta_n \to \infty$ and $\Delta_n \to 0$, as $n \to \infty$.
\end{proposition}

\begin{remark} We note the main idea for the proof is inspired by the approach in~\cite[Theorem 1.1]{NN}, which provides a bound on the some negative moments of the Malliavin derivative.

\end{remark}

\begin{proof} Let $p > 0$. Condition~\eqref{cond-karhunen} is satisfied and thus Theorem~\ref{thm-karhunen} implies that $\{Z_t, t \in \R_+ \} \overset{(d)}{=} \{ Z_t^{\xi}, t \in \R_+ \}$ where $\xi \in L^2(\R)$ with $\xi(t) = 0$, for $t \leq 0$, and $Z_t^\xi$ is given by~\eqref{rep-integ}. Then, for a positive integer $m > 2p$,
\begin{align}
\notag \E \left[ \left( \frac{1}{T} \int_0^T Z_t^2 dt \right)^{-p} \right]  =  \E \left[ \left( \frac{1}{T}\int_0^T(Z_t^{\xi})^2dt\right)^{-p}\right] = \E \left[ \left( \sum_{k =1}^{m} \frac{1}{T} \int_{(k-1)T/m}^{kT/m}(Z_t^{\xi})^2dt\right)^{-p} \right].
\end{align}
Recall, that by the inequality between arithmetic and geometric means, for any positive reals $x_1, \ldots, x_m$, $\sum_{k = 1}^m x_k \geq m \prod_{k=1}^m x_k^{1/m}$. Then,
\begin{align}
\label{ineq-0} \E \left[ \left( \frac{1}{T} \int_0^T Z_t^2 dt \right)^{-p} \right]  \leq  m^{-p} \E \left[ \prod_{k =1}^{m}\left(\frac{1}{T}\int_{(k-1)T/m}^{kT/m}(Z_t^{\xi})^2dt\right)^{-p/m}\right].
\end{align}
We proceed by conditioning. Let $\mathcal{F}_t \coloneqq \sigma(W_u, u \leq t)$. By definition, $Z_t^\xi$ is $\mathcal{F}_t-$measurable. Thus,
\begin{align}
\notag \E  \left[ \left( \frac{1}{T} \int_0^T Z_t^2 dt \right)^{-p} \right]  \leq   m^{-p} \E \left[ \prod_{k =1}^{m} \E \left[ \left(\frac{1}{T}\int_{(k-1)T/m}^{kT/m}(Z_t^{\xi})^2dt\right)^{-p/m}\Big| \mathcal{F}_{(k-1)T/m} \right] \right].
\end{align}
Next, note that
\begin{align}
 & \E \left[ \left( \frac{1}{T} \int_{(k-1)T/m}^{kT/m}(Z_t^{\xi})^2 dt \right)^{-p/m} \Big| \mathcal{F}^W_{(k-1)T/m} \right] \nonumber\\
= & \quad  \int_0^{\infty} \mathbb{P} \left(\frac{1}{T}\int_{(k-1)T/m}^{kT/m}(Z_t^{\xi})^2dt\leq x^{-m/p} \Big| \mathcal{F}^W_{(k-1)T/m}\right)dx\nonumber\\
\leq & \quad 1+\int_1^{\infty} \mathbb{P} \left(\frac{1}{T}\int_{(k-1)T/m}^{kT/m}(Z_t^{\xi})^2dt\leq x^{-m/p}\Big| \mathcal{F}^W_{(k-1)T/m} \right) dx.\label{ineq-1}
\end{align}
Applying the Carbery-Wright Inequality~\cite{CW}, there is a universal constant $c> 0$ such that, for any $\varepsilon>0$ we can write
\begin{align}
\mathbb{P}\left( \frac{1}{T} \int_{(k-1)T/m}^{kT/m} (Z_t^{\xi})^2 dt \leq \varepsilon \Big| \mathcal{F}^W_{(k-1)T/m}\right) \leq \frac{c\sqrt{\varepsilon}}{\E \left[\frac{1}{T}\int_{(k-1)T/m}^{kT/m}(Z_t^{\xi})^2dt \mid\mathcal{F}^W_{(k-1)T/m} \right]}.\label{ineq-2}
\end{align}
Next, note that for any $0 \leq a < b$,
\begin{align}
\notag \E  \left[  \int_a^b (Z_t^\xi)^2 dt \Big| \mathcal{F}_a \right] =  &\int_a^b \E \left[ (Z_t^\xi)^2  | \mathcal{F}_a \right] dt \geq \int_a^b 2 Z_a^\xi \E[ Z_t^\xi - Z_a^\xi | \mathcal{F}_a] + \E[  (Z_t^\xi - Z_a^\xi)^2 | \mathcal{F}_a]dt\\
\label{conditional-id-1} \geq & \int_a^b \int_a^t \xi^2(t-u) du  dt = \int_0^{b-a} (b-a-v) \xi^2(v)  dv,
\end{align}
where we have used It\^o isometry and the fact that $Z_t^\xi - Z_a^\xi$ is independent of $\mathcal{F}_a$.

By isometry, $\int_0^{\infty}\xi^2(v) dv = \E[Z_0^2]=1$, so there is $T_0>0$ such that $\int_0^{\frac{T_0}{2m}}\xi^2(v)dv \geq \frac{1}{2}$. Thus, by~\eqref{conditional-id-1}, for every $T \geq T_0$,
\begin{eqnarray}
\label{ineq-3}\E \left[\frac{1}{T}\int_{(k-1)T/m}^{kT/m}(Z_t^{\xi})^2dt \mid\mathcal{F}_{(k-1)T/m}\right] \geq \frac{1}{2m} \int_0^{\frac{T}{2m}} \xi^2(v) dv \geq \frac{1}{4m}.
\end{eqnarray}

Therefore, combining~\eqref{ineq-1},~\eqref{ineq-2} and~\eqref{ineq-3}, we obtain for every $T\geq T_0$,
\begin{eqnarray}
\E\left[\left(\frac{1}{T}\int_{(k-1)T/m}^{kT/m}Z_t^2dt\right)^{-p/m}\Big| \mathcal{F}_{(k-1)T/m}\right] &\leq&1+ 4c m\int_1^{\infty}x^{-\frac{m}{2p}}dx.\label{ineq-4}
\end{eqnarray}
Hence,
\begin{eqnarray}
\gamma_{m,T_0} \coloneqq \sup_{\substack{T\geq T_0 \\ 1\leq k\leq m}} \E\left[\left(\frac{1}{T}\int_{(k-1)T/m}^{kT/m}Z_t^2dt\right)^{-p/m}\Big| \mathcal{F}_{(k-1)T/m}\right] <\infty.\label{ineq-5}
\end{eqnarray}
Consequently, it follows from~\eqref{ineq-0} and~\eqref{ineq-5} that, for all $T\geq T_0$,
\begin{align}
\notag \E \left[ \left(\frac{1}{T} \int_0^T Z_t^2 dt \right)^{-p} \right] \leq m^{-p} (\gamma_{m,T_0})^m<\infty,
\end{align}
which completes the proof of  the first part of~\eqref{average-cont}. To establish the second part of~\eqref{average-cont} one need only replace $Z_t$ by $Z_t - e^{\theta t}Z_0$ in the proof above. Indeed, the key inequalities~\eqref{ineq-0},~\eqref{ineq-1} and~\eqref{ineq-2} are the same (with the change $Z_t^\xi \to Z_t^\xi - e^{\theta t}Z_0^\xi$) and the equivalent of~\eqref{conditional-id-1} is
\begin{align}
\notag \E  \left[  \int_a^b (Z_t^\xi - e^{\theta t}Z_0^\xi)^2 dt \Big| \mathcal{F}_a \right]  \geq & \int_a^b 2 (Z_a^\xi - e^{\theta t} Z_0^\xi) \E[ Z_t^\xi - Z_a^\xi | \mathcal{F}_a] + \E[  (Z_t^\xi - Z_a^\xi)^2 | \mathcal{F}_a]dt\\
\label{conditional-id-2}   \geq & \int_a^b \int_a^t \xi^2(t-u) du  dt = \int_0^{b-a} (b-a-v) \xi^2(v)  dv,
\end{align}

Now, let us prove the discrete version~\eqref{average-disc}. First we prove it for the case $n=m^2$. Define, for every $m \geq 1$, $T_m\coloneqq m^2\Delta_m$ such that $\Delta_m\rightarrow0$, and $\frac{T}{m}=m\Delta_m\rightarrow\infty$ as $m\rightarrow\infty$. Fix $p>0$, and let $m_0$ be a positive integer such that, for every $m\geq m_0$, $m > 2p$ and $\int_0^{\frac{T_m}{m}}\xi^2(v) dv \geq \frac{1}{2}$. We have $t_i = i \Delta_m$ for $i = 1, \ldots, m$. Write
\begin{eqnarray*}
\frac{1}{m^2} \sum_{i =1}^{m^2} Z_{t_{i}}^{2} &=&\frac{1}{T_m}\int_0^{T_m} Y_t dt,
\end{eqnarray*}
where $Y_t \coloneqq \sum_{i =1}^{m^2} Z_{t_{i}} 1_{(t_{i-1},t_{i}]}(t)$. Also, denote $Y_t^\xi    \coloneqq  \sum_{i =1}^{m^2} Z_{t_{i}}^\xi 1_{(t_{i-1},t_{i}]}(t)$.

We follow the same techniques as in the proof of~\eqref{average-cont}. Notice that the inequalities~\eqref{ineq-0}-\eqref{conditional-id-1} hold with $Y_t^\xi$ instead of $Z_t^\xi$.

Thus, it suffices to prove the following equivalent of~\eqref{ineq-3}: for every $m\geq m_0$,
\begin{eqnarray*}
\E\left[\frac{1}{T_m}\int_{(k-1)T_m/m}^{kT_m/m}(Y_t^{\xi})^2dt \Big| \mathcal{F}_{(k-1)T_m/m}\right] \geq  \frac{1}{2m^2}.
\end{eqnarray*}
Notice that for every $k=1, \ldots, m$,
\begin{eqnarray*}
 \frac{1}{T_m}\int_{(k-1)T_m/m}^{kT/m}(Y_t^{\xi})^2dt =  \frac{1}{m^2}\sum_{j  =1}^{m}(Z_{t_{(k-1)m+j}}^{\xi})^2.
\end{eqnarray*}
Moreover,
\begin{eqnarray*}
\E\left[(Z_{t_{(k-1)m+j}}^{\xi})^2 \Big| \mathcal{F}_{t_{(k-1)m}}\right] \geq \int_{t_{(k-1)m}}^{t_{(k-1)m+j}}\xi^2(t_{(k-1)m+j}-u)du = \int_0^{t_{j}}\xi^2(v)dv.
\end{eqnarray*}
Therefore, for every $m\geq m_0$,
\begin{eqnarray*}
\E\left[\frac{1}{T_m}\int_{(k-1)T_m/m}^{kT_m/m}(Y_t^{\xi})^2dt \mid\mathcal{F}_{(k-1)T_m/m}\right]&\geq& \frac{1}{m^2}\sum_{j  =1}^{m}\int_0^{t_{j}}\xi^2(v)dv \geq \frac{1}{m^2} \int_0^{t_{m}}\xi^2(v)dv  \\
 &=& \frac{1}{m^2} \int_0^{\frac{T}{m}}\xi^2(v)dv \geq \frac{1}{2m^2},
\end{eqnarray*}
  which yields the proof of~\eqref{average-disc} for $n=m^2$.

 For general $n$ a simple computation yields
\begin{align}
\frac{1}{n}\sum_{i =1}^{n} Z_{t_{i}}^{2} \geq \frac{1}{n} \sum_{i =1}^{\lfloor \sqrt{n} \rfloor^2} Z_{t_{i}}^{2} \geq C \frac{1}{\lfloor \sqrt{n} \rfloor^2}  \sum_{i =1}^{\lfloor \sqrt{n} \rfloor^2} Z_{t_{i}}^{2},
\end{align}
for some absolute constant $C > 0$, and thus the first part of~\eqref{average-disc} is established. The second part follows using the same techniques as above and~\eqref{conditional-id-2}.

\end{proof}

\end{document}